\documentclass[11pt]{amsart}
\usepackage{amsmath,amsthm,amssymb,amsfonts, esint}
\usepackage{url}
\usepackage{graphicx}
\usepackage{color}
\usepackage{times}
\usepackage{cite}
\usepackage{enumerate}
\usepackage{latexsym}
\usepackage{verbatim}
\usepackage{tikz,tikz-cd}
\usepackage{url}
\usepackage{enumerate}
\usepackage[bottom]{footmisc}
\usepackage{caption}
\usepackage{subcaption}
\usepackage{tabularx}
\usepackage{verbatim}
\usepackage[top=2cm, bottom=2.8cm, left=3cm, right=3cm]{geometry}%

\usepackage[nobysame,abbrev,alphabetic]{amsrefs}
 
\newtheorem{thm}{Theorem}[section]
\newcommand{\bt}{\begin{thm}}
\newcommand{\et}{\end{thm}}

\newtheorem{cor}[thm]{Corollary}  
\newcommand{\bc}{\begin{cor}}
\newcommand{\ec}{\end{cor}}

\newtheorem{lem}[thm]{Lemma}   
\newcommand{\bl}{\begin{lem}}
\newcommand{\el}{\end{lem}}

\newtheorem{prop}[thm]{Proposition}
\newcommand{\bp}{\begin{prop}}
\newcommand{\ep}{\end{prop}}

\newtheorem{defn}[thm]{Definition}
\newcommand{\bd}{\begin{defn}}    
\newcommand{\ed}{\end{defn}}

\newtheorem{rmrk}[thm]{Remark}

\newtheorem{ntn}[thm]{Notation}

\newcommand{\R}{\mathbb{R}}
\newcommand{\N}{\mathbb{N}}

\DeclareMathOperator{\Vol}{Vol}
\DeclareMathOperator{\Diam}{Diam}
\DeclareMathOperator{\Area}{Area}

\def\Xint#1{\mathchoice
{\XXint\displaystyle\textstyle{#1}}%
{\XXint\textstyle\scriptstyle{#1}}%
{\XXint\scriptstyle\scriptscriptstyle{#1}}%
{\XXint\scriptscriptstyle\scriptscriptstyle{#1}}%
\!\int}
\def\XXint#1#2#3{{\setbox0=\hbox{$#1{#2#3}{\int}$ }
\vcenter{\hbox{$#2#3$ }}\kern-.6\wd0}}

\def\dashint{\Xint-}
\graphicspath{{images/}}

\begin{document}

\title{Stability of the Positive Mass and Torus Rigidty Theorems Under Integral Curvature Bounds}
\author{Brian Allen}
\address[]{Lehman College, CUNY}
\urladdr{\url{https://sites.google.com/view/brian-allen}}

\author{Edward Bryden}
\address{Universiteit Antwerpen}
\urladdr{\url{https://etbryden.com/}}

\author{Demetre Kazaras}
\address{Duke University}
\urladdr{\url{https://sites.google.com/view/dkazaras-homepage}}

\date{\today}

\begin{abstract}
  Work of D. Stern \cite{Stern-19} and Bray-Kazaras-Khuri-Stern \cite{BKKS19} provide differential-geometric identities which relate the scalar curvature of Riemannian $3$-manifolds to global invariants in terms of harmonic functions. These quantitative formulas are useful for stability results \cites{Stern-19, KKL-2021} and show promise for more applications of this type. 
  In this paper, we analyze harmonic maps to flat model spaces in order to address conjectures concerning the geometric stability of the positive mass theorem and the Geroch conjecture. By imposing integral Ricci curvature and isoperimetric bounds, we leverage the previously mentioned formulas to establish strong control on these harmonic maps. When the mass of an asymptotically flat manifold is sufficiently small or when a Riemannian torus has almost non-negative scalar curvature, we upgrade the maps to diffeomorphisms and give quantitative H{\"o}lder closeness to the model spaces.
\end{abstract}

\maketitle 
\section{Introduction}

The geometric control which follows from integral bounds on Ricci curvature is well understood. This theory has been developed by Anderson \cite{Anderson-Ricci}, Anderson-Cheeger \cite{Anderson-Cheeger}, Colding \cites{Colding-shape,Colding-volume},  Cheeger-Colding \cite{Cheeger-Colding-1}, Dai-Wang-Zhang \cite{Dai_Wei_Zhang-2018},  Petersen \cites{Petersen-97,Petersen-2006}, Petersen-Sprouse \cite{Petersen-Sprouse-integral},  Petersen-Wei \cites{Petersen_Wei-2001,Petersen-Wei-integral1,Petersen-Wei-integral2}, Gao \cite{Gao-integral1}, and Yang \cites{DYang-integral1,DYang-integral2,DYang-integral3}. Readers unfamiliar with this convergence theory are encouraged to consult Petersen's fantastic survey article \cite{Petersen-Survey}. Speaking generally, methods in this area are based upon harmonic coordinate systems, which directly relate the metric to its Ricci curvature by an elliptic PDE. Through analysis of this equation, one is able to show that the class of manifolds with integral bounds on Ricci (or Riemann) curvature and injectivity radius bounds (or volume growth of balls) enjoys a system of harmonic coordinate charts in which the metric is controlled in Sobolev and H{\"o}lder spaces.
%$W^{2,p}$, where $p > \tfrac{n}{2}$ and $n$ denotes the manifold's dimension. 
%This special system of coordinates then lead to $C^{0,\alpha}$ compactness and geometric stability results. 
However, the arguments which obtain this control often proceed by contradiction, and the resulting metric control is generally not explicitly quantitative in terms of given geometric bounds.
%Nonetheless, these results give a thorough understanding of the convergence of Riemannian manifolds under integral curvature bounds.

In this paper we take a new approach to the harmonic coordinate method described above. On asymptotically flat $3$-manifolds and certain compact $3$-manifolds, we study three distinct harmonic functions which encode global geometric and topological information, but generally do not define a global system of coordinates. These harmonic functions are related to the scalar curvature of $3$-tori by Stern \cite{Stern-19} and the mass of asymptotically flat $3$-manifolds by Bray-Kazaras-Khuri-Stern \cite{BKKS19}. By using these breakthrough formulas in tandem with the geometric controls following from integral Ricci curvature and isoperimetric bounds, we strongly control the regularity and behavior of the harmonic functions.  Next, when the appropriate quantities are chosen to be small enough, we are able to show that these harmonic functions produce a global coordinate system. This global coordinate system is the source of the topological stability as well as quantitative H{\"o}lder stability for both the positive mass theorem and scalar torus rigidity theorem. For the positive mass theorem, this stability is explicitly computable in terms of integral Riemann curvature and isoperimetric bounds. In the case of the scalar torus rigidity theorem, the stability remains implicitly quantitative due to the use of a contradiction argument in the spirit of Petersen.

The positive mass theorem states that a complete, asymptotically flat manifold with non-negative scalar curvature has positive ADM mass. This was first proved by Schoen-Yau \cite{SchoenYauPMT} using minimal surfaces and later by Witten \cite{Witten-PMT} using spinors and the Lichnerowicz formula. The geometric stability of the positive mass theorem was first stated by Huisken-Illmanen \cite{Huisken-Ilmanen} in terms of Gromov-Hausdorff convergence and more recently by Lee-Sormani \cite{LeeSormani1} in terms of Sormani-Wenger intrinsic flat convergence. Recently, Bray-Kazaras-Khuri-Stern \cite{BKKS19} gave a new proof of the positive mass theorem in dimension $3$ with a formula which relates the mass and scalar curvature to harmonic functions defined on the asymptotically flat manifold. This formula was then used by Kazaras-Khuri-Lee \cite{KKL-2021} to show Gromov-Hausdorff convergence for a sequence of asymptotically flat manifolds with pointwise lower bounds on Ricci curvature whose mass is tending to zero. In this paper we also use this formula to obtain quantitative $C^{0,\gamma}$-stability, for some $\gamma\in(0,1)$, under integral Ricci curvature bounds and isoperimetric constants for metric balls. It is interesting to compare our methods with those of Finster-Bray \cite{Finster-Bray-99}, Finster-Kath \cite{Finster-Kath-02}, and Finster \cite{Finster-09}, which are based on a spinorial mass formula due to Witten \cite{Witten-PMT}. In their work, an isoperimetric condition is used to obtain a relation between the mass and the $W^{1,2}$- and $L^\infty$-norms of the Riemann curvature. Stability of the positive mass theorem in the asymptotically flat setting has also been studied by Allen \cite{Allen18}, Bryden \cite{Bryden20}, Bryden-Khuri-Sormani \cite{BKS}, Huang-Lee-Sormani \cite{HLS}, Huang-Lee-Perales \cite{Huang-Lee-Perales}, Lee-Sormani \cite{LeeSormani1}, Sormani-Stavrov \cite{Sormani-Stavrov-1}, and the references therein.

The main tool we use to study the geometric stability of the positive mass theorem is the following formula relating the ADM mass to scalar curvature through an asymptotically linear harmonic function.
\begin{thm}[Theorem 1.2 of \cite{BKKS19}]
    Let $(M,g)$ be an orientable complete asymptotically flat manifold with no spherical classes in $H_2(M;\mathbb{Z})$, and let $u$ be an asymptotically linear harmonic function (see Section \ref{sec-Asymptotic}). Then, the ADM mass satisfies
    \begin{equation}\label{ADM-formula}
        m(g)\geq\frac{1}{16\pi}\int_{M}\frac{|\nabla^2u|^2}{|\nabla u|}+R_{g}|\nabla u|dV_g,
    \end{equation}
    where $R_g$ denotes the scalar curvature of $(M,g)$.
\end{thm}

The first main result leverages controlled asymptotic falloff with good control on the local analysis of Sobolev functions in order to show that ADM mass strongly controls geometry. In order to effectively utilize the mass formula \eqref{ADM-formula}, we will need Sobolev and Poincar{\'e} inequalities on our manifold. Crucially, these inequalities are known to be equivalent to isoperimetric constants (See Li \cite{li_2012}). Let us now recall the isoperimetric constant central to our result.

\begin{defn}[Definition 9.2 in \cite{li_2012}]
    Let us define $IN_{\alpha}(M^{n},g)$ for $1\leq\alpha\leq\tfrac{n}{n-1}$ as follows:
    \begin{equation}
        IN_{\alpha}(M^{n},g)=\inf\left\lbrace\frac{\Area_g(S)}{\min\{\Vol_g(\Omega_1),\Vol_g(\Omega_2)\}^{\frac{1}{\alpha}}}:M=\Omega_1\cup S \cup \Omega_2,\partial \Omega_1=S=\partial\Omega_2\right\rbrace.
    \end{equation}
    $IN_{\alpha}(M^{n},g)$ is called the  Neumann $\alpha-$isoperimetric constant of $M$.
\end{defn}
With the above in hand, we can make precise what it means for a class of asymptotically flat Riemannian manifolds to have controlled asymptotic falloff and good control on the local analysis of Sobolev functions, by defining the following class of asymptotically flat Riemannian manifolds.
\begin{defn}\label{the-family-of-metrics}
    Fix $b,\bar{m},\Lambda,\kappa>0$, $\tau>\tfrac{1}{2}$, $p>1$, $\alpha\in[1,\tfrac{3}{2}]$. An oriented connected complete asymptotically flat $3$-dimensional Riemannian manifold $(M,g)$ is said to be {\emph{homologically simple, $(b,\tau,\bar{m})$ asymptotically flat, $(\Lambda,\alpha)$ Neumann-isoperimetrically bounded, and $\kappa$ curvature-$L^{p}$ bounded}} if
    \begin{enumerate}
        \item there are no spherical classes in $H_2(M;\mathbb{Z})$ and $\partial M=\emptyset$,
        \item there exists a coordinate chart $\Phi:M\setminus\Omega\rightarrow\mathbb{R}^{3}\setminus B_{1}(0)$ such that on $\mathbb{R}^{3}\setminus B_{1}(0)$ we have
        \begin{equation*}
            |D^{(k)}\left(g_{ij}-\delta_{ij}\right)|(x)\leq b|x|^{-\tau-k},
        \end{equation*}
        where $k=0,1,$ and $2$,
        \item $R_g \ge 0,$ and $m_{\mathrm{ADM}}(g)\leq\bar{m}$,
        \item for all metric balls, $B_r(x)\subset M$, we have
        \begin{equation*}
            IN_{\alpha}\left(B_r(x)\right)\geq \Lambda,
        \end{equation*}
        \item $\|\mathrm{Rc}_{g}\|_{L^{p}}\leq\kappa.$
    \end{enumerate}
    We will denote the family of such $3$-dimensional asymptotically flat Riemannian manifolds by $\mathcal{M}\left(b,\tau,\bar{m},\Lambda,\alpha,\kappa,p\right)$. In most of what follows, we will impose $\alpha=\tfrac{3}{2}$ and $p=3$. In this case, we adopt the shorthand $\mathcal{M}(b,\tau,\bar{m},\Lambda,\kappa)=\mathcal{M}\left(b,\tau,\bar{m},\Lambda,\tfrac{3}{2},\kappa,3\right)$.
\end{defn}

\begin{figure}
    \centering
    \includegraphics[width = 6cm]{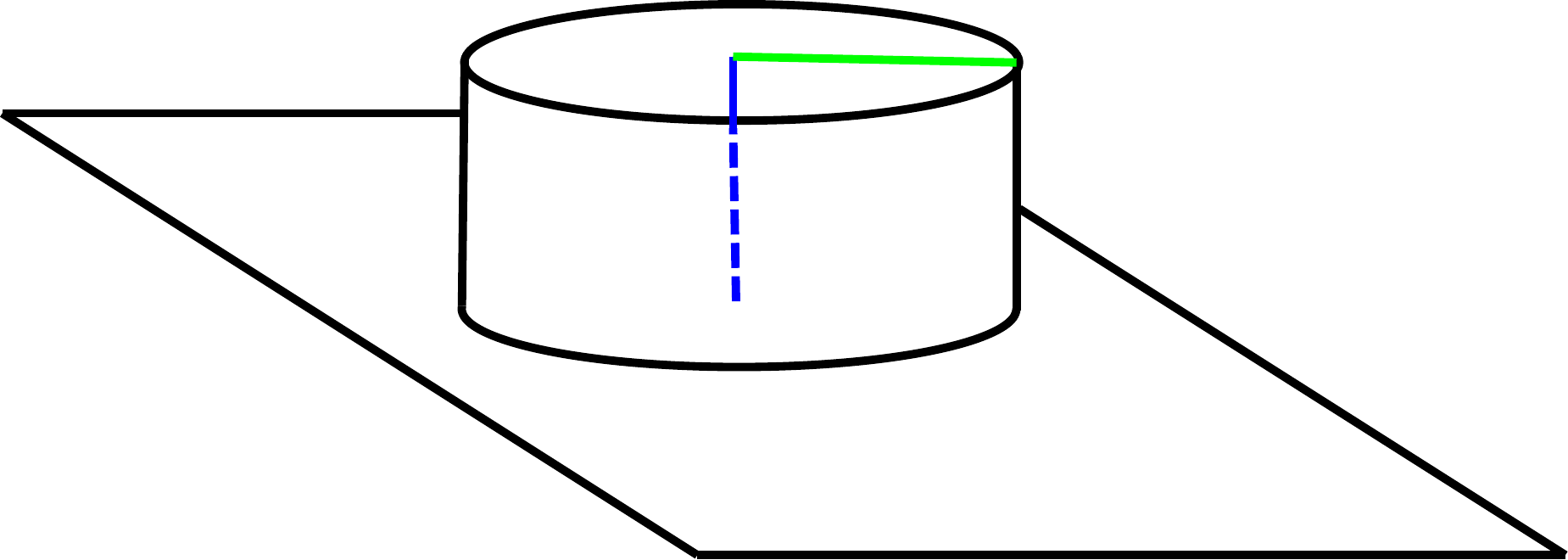}
    \caption{The above short and stout bump is Neumann $\Lambda$-isoperimetrically bounded if the ratio of its radius (in green) to its height (in blue) is bounded above and below. By taking the radius and the height to be arbitrarily small, one obtains Neumann $\Lambda$-isoperimetrically bounded spaces with arbitrarily small injectivity radius.}
    \label{fig:coin_bump}
\end{figure}

\begin{rmrk}\label{rmrk:int_Rm_from_int_Ric}
    Ricci curvature has a unique character in dimension $3$. In particular, all components of a $3$-manifold's Riemann tensor can be algebraically expressed in terms of its Ricci curvature. This means that condition $(5)$ in Definition \ref{the-family-of-metrics} is actually equivalent to a bound on $\|\mathrm{Rm}_g\|$. 
\end{rmrk}
The first main result of the paper gives quantitative stability of the positive mass theorem.
\begin{thm}\label{Thm-Mass Stability}
    Let $b,\bar m,\Lambda,\kappa>0, \tau>\tfrac12$, and $\gamma\in(0,\tfrac{1}{2})$ be given parameters. For any $\varepsilon>0$ there exists $\delta=\delta(b,\tau,\bar{m},\Lambda,\kappa,\varepsilon,\gamma)>0$ such that the following holds: if $(M,g)$ lies in $\mathcal{M}\left(b,\tau,\bar{m},\Lambda,\kappa\right)$
    and $m_{\mathrm{ADM}}(g)\leq\delta$, then $M$ is diffeomorphic to $\R^3$ and
    \begin{align}\label{e:holderstability}
        \|g-g_{\mathbb{E}}\|_{C^{0,\gamma}(\R^3)} < \varepsilon
    \end{align}
    where $g_{\mathbb{E}}$ denotes the flat metric on $\mathbb{R}^3$.
\end{thm}
\begin{rmrk}
    The H\"{o}lder norm in \eqref{e:holderstability} can be computed using distances measured with respect to either $g$ or $g_\mathbb{E}$ -- since these metrics are $C^0$-close, this choice is immaterial. 
    Also, an interesting consequence of Theorem \ref{Thm-Mass Stability} is the following topological stability: if a manifold $(M,g)\in\mathcal{M}\left(b,\tau,\bar{m},\Lambda,\kappa\right)$ has sufficiently small mass, then $M$ is diffeomorphic to $\mathbb{R}^3$.
\end{rmrk}

Next, we turn to the second main application of our methods. The Geroch conjecture asserts that a Riemannian $3$-torus with non-negative scalar curvature must be flat. This rigidity theorem was first proven by Schoen-Yau \cite{Schoen-Yau-min-surf} and soon after in higher dimensions by Gromov-Lawson \cite{Gromov-Lawson-torus} using different methods. The stability problem associated to the Geroch conjecture was first considered by Gromov \cite{GroD} and made more precise by Sormani \cite{Sormani-scalar}, who observed the necessity of an extra geometric constraint known as a {\em minA condition}. A minA condition is a lower bound on the area of all closed minimal surfaces within the manifold, and is used to eliminate pathological bubbling examples first observed by Basilio-Sormani \cite{Basilio-Sormani}. See the survey paper by Sormani \cite{ScalarSurvey-Sormani} for a recent discussion of this conjecture as well as many other conjectures involving scalar curvature. 

Various special cases of these conjectures have been studied by Allen \cite{Allen-Conformal-Torus}, Allen-Bryden \cite{AllenBryden21}, Allen-Hernandez-Vazquez-Parise-Payne-Wang \cite{AHMPPW1}, Cabrera Pacheco-Ketterer-Perales \cite{PKP19}, Chu-Lee \cite{Chu-Man-Chun22}, and Lee-Naber-Neumeyer \cite{LNN}. In this paper, we establish quantitative stability and H{\"o}lder convergence in the context of integral bounds on the Ricci curvature. We are able to avoid bubbling by requiring a uniform lower bound on the Neumann isoperimetric constant of metric balls. The present work suggests that a lower bound on the Neumann isoperimetric constant
%(for some $\alpha \in \left[1,\frac{n}{n-1}\right]$) 
is an analytically effective -- yet still geometric -- replacement to a minA condition when considering the stability of $3$-dimensional rigidity theorems related to scalar curvature. 

The main ingredient we use to establish quantitative stability of the Geroch conjecture is a formula relating the scalar curvature of a closed $3$-manifold to $\mathcal{S}^1$-valued harmonic maps developed by D. Stern. 
\begin{thm}[Theorem 1.1 of \cite{Stern-19}]
\label{t:Stern}
    Let $(M,g)$ be a closed oriented Riemannian $3$-manifold, and suppose $u:M\rightarrow\mathcal{S}^1$ is a non-constant harmonic map. Then
    \begin{equation}\label{Scalar-formula}
        \begin{split} 4\pi\int_{\mathcal{S}^1}\chi\left(\Sigma_{\theta}\right)d\theta&\geq \int_{M}\frac{|\nabla du|^2}{|du|}+R_{g}|du|dV_g.
        \end{split}
    \end{equation}
    where $\chi(\Sigma_\theta)$ is the Euler characteristic of a regular level set $\Sigma_\theta=u^{-1}(\theta)$ and $du$ denotes the $1$-form $u^*d\theta$.
\end{thm}
\begin{rmrk}
    For closed manifolds $M$, recall that there is a bijection between $H^1(M;\mathbb{Z})$ and homotopy classes of maps from $M$ to a circle. For each such homotopy class of maps, there is an harmonic representative, unique up to rotations of the circle. Consequently, the first integral cohomology group of a manifold provides the harmonic maps appearing in Theorem \ref{t:Stern}.
\end{rmrk}
%In what follows, for the torus case we will denote the Hessian of $u$ by $\nabla du$ instead of by $\nabla^2u$. The psychological benefit of this is to remember that $u$ is $\mathcal{S}^1$-valued, and that $du$ is supposed to represent a particular cohomology class.

We can now define the class of manifolds which we will consider for our second main result.

\begin{defn}\label{def:toric}
    Given parameters $\Lambda,V,\kappa>0$, we denote by $\mathcal{N}(\Lambda,V,\kappa)$ the collection of closed connected oriented Riemannian $3$-manifolds $(M,g)$ which satisfy the following:
    \begin{enumerate}
        \item $(M,g)$ is $(\Lambda,\frac{3}{2})$-Neumann-isoperimetrically bounded,
        \item $\frac{1}{V}\leq \Vol_g(M)\leq V$,
        \item $\|\mathrm{Rc}_g\|_{L^3}\leq\kappa$,
        \item the second integral homology $H_2(M;\mathbb{Z})$ contains no spherical classes,
        \item there are $3$ cohomology classes $\alpha_1,\alpha_2,\alpha_3\in H^1(M;\mathbb{Z})$ so that the cup product $\alpha_1\cup\alpha_2\cup\alpha_3$ generates $H^3(M;\mathbb{Z})$.
    \end{enumerate}
\end{defn}
\begin{rmrk}
    Let us briefly remark on the homological conditions $(4)$ and $(5)$ imposed on the underlying manifolds in the class $\mathcal{M}(\Lambda,V,\kappa)$. First, a class $\alpha\in H_2(M;\mathbb{Z})$ is said to be {\emph{spherical}} if there is an embedding $\mathcal{S}^2\hookrightarrow M$ such that $\alpha$ is the image of the fundamental class $[\mathcal{S}^2]$. The torus $T^3$ is the prototypical example satisfying $(4)$ and $(5)$, but manifolds of the form $T^3\# N^3$ where $N^3$ has the topology of a closed hyperbolic manifold or a spherical space-form also satisfy these conditions. On the other hand, manifolds with $S^2\times S^1$ summands will fail to satisfy $(4)$.
\end{rmrk}

Next, we apply our methods to investigate the stability of the Geroch conjecture. Below, we adopt the notation $R_g^-=-\min(0,R_g)$ for the negative part of scalar curvature.

\begin{thm}\label{Thm-Torus Stability}
    Let $\Lambda,V,\kappa>0$ and $\gamma\in(0,\tfrac{1}{2})$ be given. For any $\varepsilon>0$, there exists a $\delta=\delta(\Lambda,V,\kappa,\varepsilon)>0$ so that the following holds: 
    If $(M^3,g)\in\mathcal{N}(\Lambda,V,\kappa)$ and $\|R_g^-\|_{L^{1}}\leq\delta$, then $M$ is diffeomorphic to the torus $\mathbb{T}^3$ and 
    \begin{equation}
        \|g-g_{F}\|_{C^{0,\gamma}}<\varepsilon    
    \end{equation}
    for some flat metric $g_{F}$.
\end{thm}

The following sequential stability result is a straight-forward consequence of Theorem \ref{Thm-Torus Stability}.

\begin{cor}\label{Cor-SequentialTorusStability}
     Let $\Lambda,V,\kappa>0$ be given and suppose $\{(M^3_i,g_i)\}_{i=1}^\infty$ is a sequence in $\mathcal{N}(\Lambda,V,\kappa)$. If $\displaystyle\lim_{i\to\infty}\|R_{g_i}^{-}\|_{L^{1}}=0$, then $(M_i,g_i)$ subsequentially converge to a flat torus in the $C^{0,\gamma}$ topology for some $\gamma\in(0,1)$.
\end{cor}

To conclude this section, we give a brief outline of the paper. We begin in Section \ref{sec-Background} by introducing the fundamental definitions and theorems which describe volume controls for manifolds with integral bounds on Ricci curvature. In Section \ref{sec-Sobolev}, we define isoperimetric and Sobolev constants for manifolds and prove various relationships between them. At the end of this section we show that uniform control on the volume growth of balls and a $(1,p)$ Poincar{\'e} inequality are enough to guarantee a uniformly controlled constant in Morrey's inequality.

Section \ref{sec-Asymptotic} is devoted to the analysis of asymptotically linear harmonic functions on asymptotically flat manifolds. We review the asymptotic features of such functions established in \cite{KKL-2021} and develop related decay results which are needed in Sections \ref{sec-IntegralEstimates} and \ref{sec-MassEstimates}. Next, Section \ref{sec-IntegralEstimates} develops $L^2$ control on higher order derivatives of harmonic functions. These estimates are derived by integrating by parts and commuting covariant derivatives. This is an essential step where the integral Ricci bounds and asymptotic estimates show up as necessary assumptions to obtain uniform $W^{3,2}$ regularity of harmonic functions. Finally, in Section \ref{sec-MassEstimates}, we combine the mass formula, Sobolev inequalities, the asymptotic estimates, and the $W^{3,2}$ regularity of harmonic functions in order to prove Theorem \ref{Thm-Mass Stability}.

This accomplished, we turn to the proof of Theorem \ref{Thm-Torus Stability} in section \ref{sec-Torus}. We follow a fairly similar argument as in the proof of Theorem \ref{Thm-Mass Stability}. The main difference is that we cannot leverage the asymptotic control which comes from the asymptotic region of asymptotically flat manifolds. Instead, we are able to establish control on the $L^2$ norm of harmonic $1-$forms on tori and related manifolds which is used to produce the desired H{\"o}lder stability.  

%\begin{conj}\label{conj}
%Given $V,\Lambda,$ and $\varepsilon>0$, there is a $\delta=\delta(V,\Lambda,\varepsilon)>0$ so that the following holds: If $(M^3,g)$ is a closed oriented connected Riemannian $3$-manifold such that
%\begin{enumerate}
%    \item there are $3$ cohomology classes $\alpha_1,\alpha_2,\alpha_3\in H^1(M;\mathbb{Z})$ so that the cup product $\alpha_1\cup\alpha_2\cup\alpha_3$ generates $H^3(M;\mathbb{Z})$,
%    \item $V^{-1}\leq \mathrm{Vol}_g(M)\leq V$,
%    \item $IN_1(B_r(p),{\widetilde{g}}) \ge \frac{\Lambda}{r}$ for all metric balls $B_r(p)$ in the universal cover $(\widetilde{M},\widetilde{g})$,
%    \item $R_g\geq-\delta$,
%\end{enumerate}
%then $(M,g)$ is within distance $\varepsilon$ of a flat $3$-torus in %the Sormani-Wenger Intrinsic Flat topology. \footnote{See %\cite{ScalarSurvey-Sormani} for a precise definition.}
%\end{conj}

%%%%%%%%%%%%%%%%%%%%%%%%%%%%%%%%%%%%%%%%%%%%%%%%%%%%%%
\section{Background}\label{sec-Background}

% The study of the geometric estimates which follow from integral Ricci bounds has been successfully undertaken by many authors, and here we recall some results which will be used in this paper.
In this section we review some results and notions central to our strategy. We begin by recalling the following isoperimetric constant.
\begin{defn}[Definition 9.1 in \cite{li_2012}]
    Given an $n$-dimensional Riemannian manifold $(M,g)$ and a number $1\leq\alpha\leq\tfrac{n}{n-1}$, the {\emph{Dirichlet $\alpha-$isoperimetric constant of $(M,g)$}} is denoted by $ID_{\alpha}(M,g)$ and defined as
    \begin{equation}
        ID_{\alpha}(M^{n},g)=\inf\left\lbrace\frac{|\partial\Omega|}{|\Omega|^{\frac{1}{\alpha}}}:\Omega\subset M,\text{ }\partial\Omega\cap\partial M=\emptyset\right\rbrace.
    \end{equation}
\end{defn}

\noindent Happily, a lower bound on $ID_{\alpha}(M,g)$ for $\alpha>1$ gives a lower bound on the volume growth of metric balls.
\begin{lem}\label{DirichletIsoToLowerVolumeGrowth}
Let $(M,g)$ be an n-dimensional Riemannian manifold with $ID_{\alpha}(M,g)>0$ for some $\alpha\in(1,\tfrac{n}{n-1}]$. Then, we have
\begin{align}
\frac{|\partial B_r(x)|}{|B_r(x)|^{\frac{1}{\alpha}}} \ge ID_{\alpha}(M,g),
\end{align}
and, for any $r\leq \Diam(M,g)$,
\begin{align}
|B_r(x)| \ge ID_{\alpha}(M,g)^{\frac{\alpha}{\alpha-1}} r^{\frac{\alpha}{\alpha-1}}.
\end{align}
In particular, if $\alpha = \tfrac{n}{n-1}$ and $r\leq \Diam(M,g)$, then
\begin{align}
|B_r(x)| \ge ID_{\frac{n}{n-1}}(M,g)^n r^n.
\end{align}
\end{lem}
\begin{proof}
The following proof is taken from \cite{Heinonen-2001} chapter 3, page 25.
By taking the derivative of the volume of geodesic balls, we find for almost every $r > 0$
\begin{align}
\frac{d}{dr} |B_r(x)| = |\partial B_r(x)| \ge ID_{\alpha} |B_r(x)|^{\frac{1}{\alpha}},
\end{align}
and hence by integrating the ODE inequality we find
\begin{align}
|B_r(x)| \ge ID_{\alpha}^{\frac{\alpha}{\alpha-1}} r^{\frac{\alpha}{\alpha-1}}.
\end{align}
\end{proof}

Whereas lower bounds on $ID_{\alpha}(M)$ give lower bounds on the volume growth of metric balls, lower bounds on Ricci curvature give upper bounds on the volume growth of metric balls. This is the content of the following two important results due to Petersen-Wei.

\begin{lem}[Lemma 2.3 in \cite{Petersen-Wei-1997}]\label{lem:int_ric_growth}
    Given an $n$-dimensional Riemannian manifold $(M,g)$, let $h(x)$ denote the smallest eigenvalue of the Ricci endomorphism at $x\in M$, and $k(\lambda, p)$ be given by
    \begin{equation}
        k(\lambda,p)=\int_{M}\textrm{max}\left\lbrace 0,-h(x)+(n-1)\lambda\right\rbrace^{p}dV_g.
    \end{equation}
    If $\lambda\leq 0$, $r<R$, and $p>\tfrac{n}{2}$, then there is a constant $C(n,p,\lambda,R)$, such that
    \begin{equation}
        \left(\frac{|B_r(x)|}{v(n,\lambda,R)}\right)^{\frac{1}{2p}}-\left(\frac{|B_r(x)|}{v(n,\lambda,r)}\right)^{\frac{1}{2p}}\leq C\left(n,p,\lambda,R\right)\left(k(\lambda,p)\right)^{\frac{1}{2p}},
    \end{equation}
    where $v(n,\lambda,t)$ is the volume of a ball of radius $t$ in the simply connected $n$-dimensional space-form with curvature $\lambda$. 
    %Moreover, $C(n,p,\lambda,R)$ is non-decreasing in $R$.
\end{lem}
As noted in \cite{Petersen-Wei-1997}, the following is an immediate consequence.
\begin{cor}\label{volume-growth-above}
    With the same notation and assumptions of Lemma \ref{lem:int_ric_growth}, for all $0\leq r\leq R$, we have
    \begin{equation}
        |B_r(x)|\leq\left(1+C(n,p,\lambda,R)k(\lambda,p)^{\frac{1}{2p}}\right)^{2p}v\left(n,\lambda,r\right).
    \end{equation}
\end{cor}
\section{Sobolev, Poincar{\'e}, and Morrey inequalities}\label{sec-Sobolev}

In this section we remind the reader of the definitions of various Sobolev and isoperimetric constants, review well known relationships between them, and prove some new relationships which are needed in later sections. 

\begin{ntn} 
Below and throughout, if $f$ is a function defined on a domain $B$ in a Riemannian manifold $(M,g)$, we write $f_B$ for the average value $\dashint_B fdV_g:=\tfrac{1}{|B|}\int_BfdV_g$.
Secondly, without further mention, we will use $C(a_1,\dots, a_N)$ to denote a constant which depends only on parameters $a_1,\dots,a_N$. If we require the use of multiple such constants in the course of a given proof, we will use subscripts $C_k(a_1,\dots,a_N)$, $k=1,2,\dots$, to differentiate them.
\end{ntn}

Since the Neumann and Dirichlet isoperimetric constants are crucial to our forthcoming analysis, we begin with a result relating them to each other. This lemma follows from a remark in \cite{Dai_Wei_Zhang-2018}.
\begin{lem}\label{NeumanToDirichletIsoperimetricEst}
Let $(M,g)$ be an $n$-dimensional Riemannian manifold. If domains $\Omega_1  \subset \Omega_2\subset M$ satisfy $\partial \Omega_1 \cap \partial\Omega_2 = \emptyset$ and $|\Omega_1|\le \frac{1}{2} |\Omega_2|$, then
\begin{align}
    IN_{\alpha}(\Omega_2,g) \le ID_{\alpha}(\Omega_1,g).
\end{align}
Furthermore, if $\Omega_1' \subset \Omega_2'$ so that $\partial \Omega_1' \cap \partial \Omega_2' = \emptyset$ we find
\begin{align}
    ID_{\alpha}(\Omega_2',g) \le  ID_{\alpha}(\Omega_1',g).
\end{align}
\end{lem}
\begin{proof}
    Consider $U \subset \Omega_1$ so that $\partial U \cap \partial \Omega_1 = \emptyset$. Then we see that $U$ is a valid competitor for $ID_{\alpha}(\Omega_1,g)$ where 
    \begin{align}
     |U| \le |\Omega_1| \le \frac{1}{2} |\Omega_2| ,  
    \end{align} and 
    \begin{align}
       |\Omega_2 \setminus U| =|\Omega_2|-|U|\ge \frac{1}{2}|\Omega_2| .
    \end{align}
    Now we notice that $U$ is a valid competitor for $IN_{\alpha}(\Omega_2,g)$ and hence
    \begin{align}
     IN_{\alpha}(\Omega_2,g) \le \frac{|\partial U|}{\min\{|U|,|\Omega_2\setminus U|\}^{\frac{1}{\alpha}}}   \le \frac{|\partial U|}{|U|^{\frac{1}{\alpha}}}.\label{CompetitorRelationship}
    \end{align}
    Since this is true for any $U$ so that $U \subset \Omega_1$ and $\partial U \cap \partial \Omega_1 = \emptyset$ we can take the infimum over the right hand side of \eqref{CompetitorRelationship} to find the first result.
    
    The second result follows from the fact that any competitor for $ID_{\alpha}(\Omega_1',g)$ is a competitor for $ID_{\alpha}(\Omega_2',g)$.
\end{proof}

\subsection{Dirichlet and Neumann Sobolev Constants}

We start by defining the Dirichlet and Neumann Sobolev constants.

% \begin{defn}[Definition 9.1 in \cite{li_2012}]
%     Let us define $ID_{\alpha}(M^{n},g)$ for $1\leq\alpha\leq\frac{n}{n-1}$ as follows:
%     \begin{equation}
%         ID_{\alpha}(M^{n},g)=\inf\left\lbrace\frac{\Area_g(\partial\Omega)}{\Vol_g(\Omega)^{\frac{1}{\alpha}}}:\Omega\subset M,\text{ }\partial\Omega\cap\partial M=\emptyset\right\rbrace.
%     \end{equation}
%     $ID_{\alpha}(M^{n},g)$ is called the  Dirichlet $\alpha-$isoperimetric constant of $M$.
% \end{defn}

% \begin{defn}[Definition 9.2 in \cite{li_2012}]
%     Let us define $IN_{\alpha}(M^{n},g)$ for $1\leq\alpha\leq\frac{n}{n-1}$ as follows:
%     \begin{equation}
%         IN_{\alpha}(M^{n},g)=\inf\left\lbrace\frac{\Area_g(S)}{\min\{\Vol_g(\Omega_1),\Vol_g(\Omega_2)\}^{\frac{1}{\alpha}}_{g}}:M=\Omega_1\cup S \cup \Omega_2,\partial \Omega_1=S=\partial\Omega_2\right\rbrace.
%     \end{equation}
%     $IN_{\alpha}(M^{n},g)$ is called the  Neumann $\alpha-$isoperimetric constant of $M$.
% \end{defn}

\begin{defn}[Definition 9.3 in \cite{li_2012}]
    Let us define $SD_{\alpha}(M^{n},g)$ for $1\leq\alpha\leq\frac{n}{n-1}$ as follows:
    \begin{equation}
        SD_{\alpha}(M^{n},g)=\inf\left\lbrace\frac{\|\nabla f\|_{L^{1}}}{\|f\|_{L^{\alpha}}}:f\in W_0^{1,1}(M^{n},g)\right\rbrace.
    \end{equation}
    $SD_{\alpha}(M^{n},g)$ is called the Dirichlet $\alpha-$Sobolev constant of $M$.
\end{defn}

\begin{defn}[Definition 9.4 in \cite{li_2012}]
     Given an $n$-dimensional Riemannian manifold $(M,g)$, let us define $SN_{\alpha}(M^{n},g)$ for $1\leq\alpha\leq\frac{n}{n-1}$ as follows:
     \begin{equation}
         SN_{\alpha}(M,g)=\inf\left\lbrace\frac{\|\nabla f\|_{L^{1}}}{\inf_{k\in\mathbb{R}}\|f-k\|_{L^{\alpha}}}:f\in W^{1,1}(M,g)\right\rbrace
     \end{equation}
     $SN_{\alpha}(M,g)$ is called the Neumann $\alpha$-Sobolev constant of $(M,g)$.
\end{defn}

Let us now recall and observe some relationships between the above quantities. It is shown in  \cite[Theorem 9.5]{li_2012} that 
\begin{align}
  SD_{\alpha}(M,g)=ID_{\alpha}(M,g),
\end{align} 
and in \cite[Theorem 9.6]{li_2012} we see that
\begin{align}
    \min\{1,2^{\frac{\alpha-1}{\alpha}}\} IN_{\alpha}(M,g) \le SN_{\alpha}(M,g) \le \max\{1,2^{\frac{\alpha-1}{\alpha}}\} IN_{\alpha}(M,g).
\end{align}

The quantity $SN_{\alpha}$ is a bit mysterious in its present form, mainly because of the term involving $k$. However, the following result, mentioned in \cite{li_2012}, helps us understand it a little better. We will reproduce the statement and proof for the convenience of the reader.
\begin{lem}\label{Neuman-Sobolev-ineq}
    Let $(M,g)$ be an $n$-dimensional Riemannian manifold. Suppose $|M|<\infty$ and $f\in L^{q}$, $q\ge 1$. Then, there exists a unique $k_q(f)$ such that
    \begin{equation}
        \inf_{k\in\mathbb{R}}\|f-k\|_{L^{q}}=\|f-k_{q}(f)\|_{L^{q}}.
    \end{equation}
    Further, $k_q(f)$ is given by the equation
    \begin{equation}
        \int_M\mathrm{sgn}\left(f-k_{q}(f)\right)|f-k_{q}(f)|^{q-1}dV_g=0.
    \end{equation}
    In particular, when $q=2$ we have that
    \begin{equation}
        k_{2}(f)=\frac{1}{|M|}\int_{M}fdV_g=\dashint_{M}fdV_g.
    \end{equation}
\end{lem}
\begin{proof}
    Observe that for $q>1$ the function $F(k)= \int_M |f-k|^{q}dV_g$ is strictly convex and differentiable. Furthermore, observe that $\displaystyle\lim_{|k|\rightarrow\infty}\int |f-k|^q=\infty$. Therefore, the function has a unique minimal point, which is also its unique critical point. We calculate that
	\begin{equation}
	F'(k)=	\frac{d}{dk}\int_M |f-k|^qdV_g=q\int_M \text{sign}(f-k)|f-k|^{q-1}dV_g.
	\end{equation}
	In the case that $q=1$, one can observe that for $k_1,k_2 \in \R$, $t \in [0,1]$ that
	\begin{align}
	    F(k_1t+k_2(1-t)) \le tF(k_1)+(1-t)F(k_2),
	\end{align}
	and hence $F$ is convex, differentiable, $\displaystyle\lim_{|k|\rightarrow\infty}F(k)=\infty$, and with only one critical point
	\begin{equation}
	F'(k)=	\frac{d}{dk}\int_M |f-k|dV_g=\int_M \text{sign}(f-k)dV_g=|\{f \ge k\}|-|\{f\le k\}|.
	\end{equation}
\end{proof}

\subsection{Relationship to Classical Poincar{\'e} and Sobolev Inequalities}

Our goal in this subsection is to show how the Neumann and Dirichlet Sobolev constants are related to the classical Poincar{\'e} and Sobolev inequalities. We begin by recalling a result from \cite{li_2012}.

 \begin{prop}[Found in \cite{li_2012}]\label{DirichletToPoincare}
     Let $(M,g)$ be an $n$-dimensional Riemannian manifold with non-empty boundary, and suppose that $SD_{\alpha}(M,g)>0$. Then, for any $1\leq p<\infty$ we have for $f\in W^{1,p}_{0}(M)$ that
  \begin{equation}
         \|f\|_{\frac{\alpha p}{p-(p-1)\alpha}}\leq SD^{-1}_{\alpha}(M,g)\frac{p}{p-\alpha(p-1)}\|\nabla f\|_{L^{p}}.
    \end{equation}
 \end{prop}
 \begin{proof}
     For the moment, suppose we are given $f\in W_{0}^{1,p}\bigcap L^{q}$, for $q=\frac{p}{p-\alpha(p-1)}$. Then, we have that $g=f^{q}$
    \begin{equation}
         \begin{split}
             \|f\|^{\frac{p}{p-\alpha(p-1)}}_{\frac{\alpha p}{p-\alpha(p-1)}}&=\left(\int_M |g|^{\alpha}dV_g\right)^{\frac{1}{\alpha}}\leq SD^{-1}_{\alpha}\left(M\right)\int_M|\nabla g|dV_g
             \\
             &=q SD^{-1}_{\alpha}(M)\int_M|g|^{q-1}|\nabla f|dV_g
             \\
             &\leq qSD^{-1}_{\alpha}(M)\|f\|^{\frac{\alpha p}{p-\alpha(p-1)}\frac{p-1}{p}}_{\frac{\alpha p}{p-\alpha(p-1)}}\|\nabla f\|_{L^{p}}.
         \end{split}
     \end{equation}
     Rearranging the above gives
    \begin{equation}
        \|f\|_{\frac{\alpha p}{p-\alpha(p-1)}}\leq q SD^{-1}_{\alpha}(M)\|\nabla f\|_{L^{p}}.
     \end{equation}
     In order to remove the assumption that $f$ is in $L^{q}$, we note that $W_{0}^{1,p}\bigcap L^{q}$ is dense in $W_{0}^{1,p}$. 
 \end{proof}
 
 We now use this result to prove a Sobolev inequality.

\begin{thm}\label{DirichletToSobolev}
    Let $M$ be asymptotically flat and suppose that $SD_{\alpha}(B_r(x),g)\ge \Lambda >0$, $\forall x \in M, r > 0$. Then, for any $1\leq p<\infty$ we have for $f\in W^{1,p}(M)$ that
    \begin{equation}
        \|f\|_{\frac{\alpha p}{p-(p-1)\alpha}}\leq \Lambda^{-1}\frac{p}{p-\alpha(p-1)}\| f\|_{W^{1,p}}
    \end{equation}
\end{thm}
\begin{proof}
    Consider $f \in W^{1,p}(M)$ and let $\phi \in C_c^{\infty}(M)$ so that $0\le \phi \le 1$ and $|\nabla \phi|\le 1$. Then we have that $\phi f \in W_0^{1,p}(B_r(x))$ for some $x \in M, r >0$ and so by Proposition \ref{DirichletToPoincare} we find 
\begin{align}
        \|\phi f\|_{\frac{\alpha p}{p-(p-1)\alpha}}&\leq SD^{-1}_{\alpha}(B_r(x))\frac{p}{p-\alpha(p-1)}\|\nabla (\phi f)\|_{L^{p}(B_r(x)}
        \\& = SD^{-1}_{\alpha}(B_r(x))\frac{p}{p-\alpha(p-1)}\left(\|\phi \nabla  f\|_{L^{p}(B_r(x)}+\|f\nabla \phi \|_{L^{p}(B_r(x)} \right)
         \\&\le \Lambda\frac{p}{p-\alpha(p-1)}\left(\| \nabla  f\|_{L^{p}(B_r(x))}+\|f \|_{L^{p}(B_r(x))} \right)
          \\&\le \Lambda\frac{p}{p-\alpha(p-1)}\left(\| \nabla  f\|_{L^{p}(M)}+\|f \|_{L^{p}(M)} \right).
\end{align}
Now by choosing a sequence of functions $\phi_n \in C_c^{\infty}(M)$ that converge locally uniformly to $1$ so that $0\le \phi_n \le 1$, $|\nabla \phi_n| \le 1$, and supp$(\phi_n) \subset B_n(x)$ for some $x \in M$ we can apply the previous inequality to obtain the desired result
\begin{align}
      \| f\|_{\frac{\alpha p}{p-(p-1)\alpha}}= \lim_{n \rightarrow \infty}\|\phi_n f\|_{\frac{\alpha p}{p-(p-1)\alpha}}&
        \le \Lambda\frac{p}{p-\alpha(p-1)}\left(\| \nabla  f\|_{L^{p}(M)}+\|f \|_{L^{p}(M)} \right).
\end{align}
\end{proof}

We have the following proposition, which can be found, in essence, in \cite[Corollary 9.9]{li_2012}.

\begin{prop}\label{NeumantoSobolev}
    Let $(M,g)$ be an $n$-dimensional Riemannian manifold, $1<p<n$, $1<\alpha\leq\frac{n}{n-1}$,  $f \in W^{1,p}(M)$, and let $k(f)$ be the constant such that
    \begin{equation}\label{eq:Sobolev-Neumann-constant}
        \int_M\mathrm{sign}\left(f-k(f)\right)\left|f-k(f)\right|^{\frac{(\alpha-1)p}{p-\alpha(p-1)}}dV_g=0.
    \end{equation}
    Then, we have
    \begin{equation}
        \begin{split}
            &\|f-k(f)\|_{L^{\frac{\alpha p}{p-(p-1)\alpha}}}\leq\frac{p}{p-(p-1)\alpha}SN^{-1}_{\alpha}(M)\|\nabla f\|_{L^{p}}
            \\
            &\|f\|_{L^{\frac{\alpha p}{p-\alpha(p-1)}}}\leq C\left(p,\alpha,SN_{\alpha}(M),|M|\right)\|f\|_{W^{1,p}}
        \end{split}
    \end{equation}
\end{prop}
\begin{proof}
    Let $f\in W^{1,p}$ and assume for the moment that $f\in L^{\alpha q}$, for some $q$ to be determined. Let $k(f)$ be the unique number such that
    \begin{equation}
        \int_M\mathrm{sign}\left(f-k(f)\right)\left|f-k(f)\right|^{q(\alpha-1)}dV_g=0.
    \end{equation}
    Let us define a function $\psi$ as follows
    \begin{equation}
        \psi=\mathrm{sign}\left(f-k(f)\right)\left|f-k(f)\right|^{q}.
    \end{equation}
    Then, by our choice of $k(f)$, we see that 
    \begin{equation}
        \int_M\mathrm{sign}\left(\psi\right)|\psi|^{\alpha-1}dV_g=0,
    \end{equation}
    and so
    \begin{equation}
        \inf_{k\in\mathbb{R}}\int_M |\psi-k|^{\alpha}dV_g=\int_M |\psi|^{\alpha}dV_g.
    \end{equation}
    
    Let us now apply the Sobolev-Neumann inequality to $\psi$ to obtain
    \begin{equation}
        \left(\int_M|\psi|^{\alpha}dV_g\right)^{\frac{1}{\alpha}}\leq SN^{-1}_{\alpha}\int_M \left|\nabla \psi\right|dV_g.
    \end{equation}
    Plugging in our formula for $\psi$, 
    \begin{equation}
        \left(\int_M|f-k(f)|^{q\alpha}dV_g\right)^{\frac{1}{\alpha}}\leq q SN^{-1}_{\alpha}\int_M|f-k(f)|^{q-1}\left|\nabla f\right|dV_g.
    \end{equation}
    Applying H{\"o}lder's inequality to the right hand side leads to 
    \begin{equation}
        \left(\int_M|f-k(f)|^{q\alpha}dV_g\right)^{\frac{1}{\alpha}}\leq q SN^{-1}_{\alpha}\left(\int_M|f-k(f)|^{\frac{(q-1)p}{p-1}}dV_g\right)^{\frac{p-1}{p}}\|\nabla f\|_{L^{p}}.
    \end{equation}
    In order to match exponents on the left and the right of the above inequality, we choose $q$ to be the solution to
    \begin{equation}
        q\alpha=\frac{(q-1)p}{p-1},
    \end{equation}
    which is $q=\frac{p}{p-\alpha(p-1)}$. Note that because $p<n$ we have $\tfrac{n}{n-1}<\frac{p}{p-1}$, so $q$ is well defined. Plugging this value in for $q$ shows that if $f$ is in $L^{\alpha q}$, then we have
    \begin{equation}\label{FirstNeumanIneq}
        \|f-k(f)\|_{L^{\frac{\alpha p}{p-\alpha(p-1)}}}\leq SN^{-1}_{\alpha}\frac{p}{p-\alpha(p-1)}\|\nabla f\|_{L^{p}}.
    \end{equation}
    
    Still assuming that $f$ is in $L^{\frac{\alpha p}{p-\alpha(p-1)}}$, we get that
    \begin{equation}\label{first-improved-integrability}
        SN^{-1}_{\alpha}\frac{p}{p-\alpha(p-1)}\|\nabla f\|_{L^{p}}+|M|^{\frac{p-\alpha(p-1)}{p\alpha}}|k(f)|\geq \|f\|_{L^{\frac{\alpha p}{p-\alpha(p-1)}}}.
    \end{equation}
    In fact, since $\alpha>1$, we can estimate $|k(f)|$ in terms of $\|f\|_{W^{1,p}}$. This is done as follows: We begin with the estimate
    \begin{equation}
        |M||k(f)|=\int_{M}|k(f)|dV_g\le\int_{M}|k(f)-f|+|f|dV_g.
    \end{equation}
    We may now use H{\"o}lder's inequality twice to obtain
    \begin{equation}
        |M||k(f)|\leq|M|^{\frac{2\alpha p-p-\alpha}{\alpha p}}\|f-k(f)\|_{\frac{\alpha p}{p-(p-1)\alpha}}+|M|^{\frac{p-1}{p}}\|f\|_{L^{p}}.
    \end{equation}
    At this point, we can apply \eqref{FirstNeumanIneq} to obtain
    \begin{equation}
        |M||k(f)|\leq |M|^{\frac{2\alpha p-p-\alpha}{\alpha p}}\frac{p}{p-(p-1)\alpha}SN^{-1}_{\alpha}(M)\|\nabla f\|_{L^{p}}+|M|^{\frac{p-1}{p}}\|f\|_{p}.
    \end{equation}
    We can now put everything together to get
    \begin{equation}
        \|f\|_{\frac{\alpha p}{p-(p-1)\alpha}}\leq C(p,\alpha,SN_{\alpha},|M|)\|f\|_{W^{1,p}}
    \end{equation}
    so long as $f$ is in $L^{\alpha q}$. Since it is well known that $W^{1,p}\bigcap L^{\alpha q}$ is dense in $W^{1,p}$, this completes the proof.

\end{proof}

We now establish that a bound on the Neumann isoperimetric constant of balls gives a Poincar{\'e} inequality.
\begin{lem}\label{Poincare-from-uniform-SN}
    Let $(M,g)$ be a $n$-dimensional Riemannian manifold. Suppose there exists a $\Lambda>0$ such that, for any ball $B_r(x)$, we have $SN_{\tfrac{n}{n-1}}(B_r(x),g)\geq\Lambda$. Then, for any $f$ in $W^{1,p}(M)$ we have
    \begin{equation}
        \left(\dashint_{B_r(x)}|f-f_{B_r(x)}|^{n}dV_{g}\right)^{\frac{1}{n}}\leq(n-1)\Lambda^{-1}|B_r(x)|^{\frac{1}{n}}\left(\dashint_{B_r(x)}|\nabla f|^{\frac{n}{2}}dV_{g}\right)^{\frac{2}{n}},
    \end{equation}
    from which it follows that
    \begin{equation}
        \dashint_{B_r(x)}|f-f_{B_r(x)}|dV_{g}\leq (n-1)\Lambda^{-1}|B_r(x)|^{\frac{1}{n}}\left(\dashint_{B_r(x)}|\nabla f|^{p}dV_{g}\right)^{\frac{1}{p}}
    \end{equation}
    for any $p\geq\frac{n}{2}$.
\end{lem}
\begin{proof}
    The starting point is to make sure that the constant $k(f)$ in Equation \eqref{eq:Sobolev-Neumann-constant} is 
    \begin{equation}
        f_{B_r(x)}=\dashint_{B_r(x)}fdV_{g}.
    \end{equation}
    Inspecting the exponent in Equation \eqref{eq:Sobolev-Neumann-constant}, while keeping in mind that $\alpha=\tfrac{n}{n-1}$, shows that we need to pick $p$ such that
    \begin{equation}
        1=\frac{(\alpha-1)p}{p-\alpha(p-1)}=\frac{p}{n-p}.
    \end{equation}
    That is, we have $p=\tfrac{n}{2}$. So, using this exponent in Proposition \ref{NeumantoSobolev}, we get
    \begin{equation}
        \left(\int_{B_r(x)}|f-f_{B_r(x)}|^{n}dV_{g}\right)^{\frac{1}{n}}\leq (n-1)\Lambda^{-1}\left(\int_{B_r(x)}|\nabla f|^{\frac{n}{2}}dV_{g}\right)^{\frac{2}{n}}.
    \end{equation}
    Now, we can rewrite this in terms of averages as follows:
    \begin{equation}
        \left(\dashint_{B_r(x)}|f-f_{B_r(x)}|^{n}dV_{g}\right)^{\frac{1}{n}}|B_r(x)|^{\frac{1}{n}}\leq(n-1)\Lambda^{-1}|B_r(x)|^{\frac{2}{n}}\left(\dashint_{B_r(x)} |\nabla f|^{\frac{n}{2}}dV_{g}\right)^{\frac{n}{2}}.
    \end{equation}
    Dividing both sides by $|B_r(x)|^{\frac{1}{n}}$ gives the first result.
    
    To get the second result, we can begin by applying H{\"o}lder's inequality to get
    \begin{equation}
        \begin{split}
            \int_{B_r(x)}|f-f_{B_r(x)}|dV_{g}&\leq|B_r(x)|^{\frac{n-1}{n}}\left(\int_{B_r(x)}|f-f_{B_r(x)}|^{n}dV_{g}\right)^{\frac{1}{n}}
            \\
            &\leq|B_r(x)|^{\frac{n-1}{n}}(n-1)\Lambda^{-1}\left(\int_{B_r(x)}|\nabla f|^{\frac{n}{2}}dV_{g}\right)^{\frac{2}{n}}
            \\
            &\leq|B_r(x)|^{\frac{n-1}{n}}(n-1)\Lambda^{-1}|B_r(x)|^{\frac{2(\sigma-1)}{n\sigma}}\left(\int_{B_r(x)}|\nabla f|^{\frac{n\sigma}{2}}dV_{g}\right)^{\frac{2}{n\sigma}}
            \\
            &=(n-1)\Lambda^{-1}|B_r(x)|^{1+\frac{1}{n}}\left(\dashint_{B_r(x)}|\nabla f|^{\frac{n\sigma}{2}}dV_{g}\right)^{\frac{2}{n\sigma}}
        \end{split}
    \end{equation}
    So, if we divide out by $|B_r(x)|$ and let $p=\frac{n\sigma}{2}$, we get the second result.
\end{proof}

\subsection{Morrey's Inequality}
In this subsection we would like to understand what conditions are needed to conclude a uniform Morrey inequality from a uniform Sobolev inequality. We start by defining a local H\"{o}lder norm for functions which will be advantageous when implying a Morrey inequality from a Poincar{\'e} inequality.

\begin{defn}
    Let $(M,g)$ be a Riemannian manifold. The H\"{o}lder norm of scale $R>0$ and power $0<\gamma < 1$ is defined as follows:
    \begin{align}
    \|u\|_{C_R^{0,\gamma}(M,g)} =\sup_{x\in M}|u(x)|+\sup_{x,y\in M, d_g(x,y)<R}\frac{|u(x)-u(y)|}{d_g(x,y)^{\gamma}}.
\end{align}
\end{defn}
\begin{prop}\label{prop:equiv-Holder-norms}
     Let $(M,g)$ be a Riemannian manifold. For any radius $R>0$ and function $u$, we have
     \begin{equation}
         \|u\|_{C_{R}^{0,\gamma}(M,g)}\leq\|u\|_{C^{0,\gamma}(M,g)}\leq\left(1+\frac{2}{R^{\gamma}}\right)\|u\|_{C_{R}^{0,\gamma}(M,g)}.
     \end{equation}
\end{prop}
\begin{proof}
    The first inequality 
    %follows from letting $R$ tend to infinity in 
    follows quickly from the definition of $\|u\|_{C_{R}^{0,\gamma}(M,g)}$. For the second inequality, for any $x,y$ in $M$ we have one of two cases. Either $d_{g}(x,y)<R$ or $d_{g}(x,y)\geq R$. In the first case, we have
    \begin{equation}
        \frac{|u(x)-u(y)|}{d_{g}(x,y)^{\gamma}}\leq\sup_{x,y\in M,d_{g}(x,y)<R}\frac{|u(x)-u(y)|}{d_{g}(x,y)^{\gamma}},
    \end{equation}
    and in the second case, we have
    \begin{equation}
        \frac{|u(x)-u(y)|}{d_{g}(x,y)^{\gamma}}\leq\frac{|u(x)|+|u(y)|}{R^{\gamma}}\leq\frac{2}{R^{\gamma}}\sup_{x\in M}|u(x)|.
    \end{equation}
\end{proof}

Now we show that a uniform, local Poincar{\'e} inequality combined with volume growth bounds for small balls implies a uniform Morrey inequality exists. This is the analytical backbone of both Theorem \ref{Thm-Mass Stability} and Theorem \ref{Thm-Torus Stability}.

\begin{thm}\label{SobolevToMorrey}
Let $(M,g)$ be an $n$-dimensional Riemannian manifold. Fix $p > 1$ and $R>0$. Assume there are constants $C_P$ and $C_V$ so that the following hold for all $r\in(0,2R)$ and $x\in M$:
\begin{align}
\dashint_{B_r(x)} |u-u_{B_r(x)}|dV_g \le C_P \Diam(B_r(x))^{\mu}\left( \dashint_{B_r(x)} |\nabla u|^p dV_g \right)^{\frac{1}{p}},
\end{align}
for any $u\in W^{1,p}(M,g)$ and
\begin{align}
C_V^{-1} r^n\le |B_r(x)| \le C_V r^n\label{LowerVolumeBound}.
\end{align}
Then, for all $u \in C^1(M)$, points $x,y \in M$ such that $d_g(x,y) \le R$, and $\mu > \tfrac{n}{p}$, we have
\begin{align}
\frac{|u(x)-u(y)|}{d_g(x,y)^{\frac{p\mu-n}{p}}} \le C(p,\mu, C_V,C_P) \left( \int_{M} |\nabla u|^p dV_g \right)^{\frac{1}{p}},
\end{align} 
and
\begin{align}
    \|u\|_{C_R^{0,\gamma}(M,g)} \le C(p,\mu, C_V,C_P)\left(\int_M|u|dV_g+\left( \int_{M} |\nabla u|^p dV_g \right)^{\frac{1}{p}}\right),
\end{align}
where $\gamma = \frac{p\mu-n}{p}$.
\end{thm}
\begin{proof}
Here we more or less follow the proof given by N. Shanmugalingam in \cite{Shanmugalingam-2000} where our goal is to adapt the argument to our setting and keep careful track of the constants.
Since $u \in C^1(M)$ we know that every $x \in M$ is a Lebesgue point of $u$ and hence
\begin{align}
\lim_{r \rightarrow 0^+} \dashint_{B_r(x)} |u(x)-u(y)| dV_g(y) = 0.
\end{align}
Since  
\begin{align}
\left| \dashint_{B_r(x)} u dV_g - u(x) \right| = \left| \dashint_{B_r(x)} u(y)-u(x) dV_g(y) \right|  \le  \dashint_{B_r(x)} |u(x)-u(y)| dV_g(y),
\end{align}
we see that
\begin{align}
\lim_{r \rightarrow 0^+} \dashint_{B_r(x)}u dV_g = u(x).
\end{align}
Now, for $x,y \in M$ such that $d_g(x,y)<R$, we define nested families of balls $\left\{B_i \right\}_{-\infty}^{\infty}$ where
\begin{align}
B_0=B_{2d_g(x,y)}(x) , \quad B_i=B_{2^{1-i}d_g(x,y)}(x), \quad B_{-i}=B_{2^{1-i}d_g(x,y)}(y), \quad i \in \N.
\end{align}
Now if we let $u_{B_i}=\dashint_{B_i} u dV_g$ then we notice that
\begin{align}
|u(x)-u(y)|&=\lim_{r \rightarrow 0^+} \left| \dashint_{B_r(x)}u dV_g-\dashint_{B(y,r)}u dV_g\right|
\\&=\lim_{i \rightarrow \infty} \left| \dashint_{B_{2^{1-i}d_g(x,y)}(x)}u dV_g-\dashint_{B_{2^{1-i}d_g(x,y)}(y)}u dV_g\right|
\\&=\left| \sum_{i=0}^{\infty}\left(u_{B_{i+1}}-u_{B_{i}}\right) -\sum_{i=0}^{\infty}\left(u_{B_{-i-1}}-u_{B_{-i}}\right)\right|
\\&= \left|\sum_{-\infty}^{\infty}\left(u_{B_{i+1}}-u_{B_{i}}\right) \right| \le \sum_{-\infty}^{\infty}|u_{B_{i+1}}-u_{B_{i}}|.\label{LastLineLebesgueSum}
\end{align}
Rewriting the terms on the right hand side of \eqref{LastLineLebesgueSum}
\begin{align}
|u_{B_{i+1}}-u_{B_i}|& = \left| \dashint_{B_{i+1}}u dV_g-\dashint_{B_i}u dV_g  \right|
\\&= \left|\dashint_{B_{i+1}} u - u_{B_i} dV_g  \right|
\\&\le \dashint_{B_{i+1}} |u - u_{B_i}| dV_g \le \frac{|B_i|}{|B_{i+1}|} \dashint_{B_{i}} |u - u_{B_i}| dV_g.
\end{align}
On the other hand, by using the volume growth assumption, the ratios of volumes of balls satisfy
\begin{align}
 \frac{|B_i|}{|B_{i+1}|} & \le \frac{(C_V 2^{1-i}d_g(x,y))^n}{(C_V^{-1}2^{-i}d_g(x,y))^n}\le C_V^{2n} 2^n    \quad i \ge 0,
\\  \frac{|B_i|}{|B_{i+1}|} & \le \frac{(C_V 2^{2+i}d_g(x,y))^n}{(C_V^{-n}2^{1+i}d_g(x,y))^n}\le C_V^{2n} 2^{-n} \le C_V^{2n} 2^n    \quad i < 0.
\end{align}
Combining the above inequalities yields
\begin{align}
    |u_{B_{i+1}}-u_{B_i}|& \le C_V^{2n} 2^n  \dashint_{B_{i}} |u - u_{B_i}| dV_g.
\end{align}

Next, we apply the Poincar{\'e} inequality to find
\begin{align}
|u_{B_{i+1}}-u_{B_i}|&\le (2C_V)^{2n} C_P \Diam(B_i)^{\mu}\left( \dashint_{B_i} |\nabla u|^p dV_g \right)^{\frac{1}{p}}
\\&\le \frac{(2C_V)^{2n}C_P\Diam(B_i)^{\mu}}{|B_i|^{\frac{1}{p}}}\left( \int_{B_i} |\nabla u|^p dV_g \right)^{\frac{1}{p}}
\\&\le \frac{(2C_V)^{2n}2^{\mu}C_P2^{\mu(1-|i|)}d_g(x,y)^{\mu}}{(C_V 2^{1-|i|}d_g(x,y))^{\frac{n}{p}}}\left( \int_{B_i} |\nabla u|^p dV_g \right)^{\frac{1}{p}}
\\&\le (2C_V)^{2n}2^{\mu}C_PC_V^{-n/p}2^{\frac{p\mu-n}{p}(1-|i|)}d_g(x,y)^{\frac{p\mu-n}{p}}\left( \int_{B_i} |\nabla u|^p dV_g \right)^{\frac{1}{p}}.
\end{align}
Putting everything together,
\begin{align}
|u(x)-u(y)|&\le \sum_{-\infty}^{\infty}(2C_V)^{2n}2^{\mu}C_PC_V^{-n/p}2^{\frac{p\mu-n}{p}(1-|i|)}d_g(x,y)^{\frac{p\mu-n}{p}}\left( \int_{B_i} |\nabla u|^p dV_g \right)^{\frac{1}{p}}
\\&\le (2C_V)^{2n}2^{\mu}C_PC_V^{-n/p} d_g(x,y)^{\frac{p\mu-n}{p}}\left( \int_{M} |\nabla u|^p dV_g \right)^{\frac{1}{p}} \sum_{-\infty}^{\infty}2^{\frac{p\mu-n}{p}(1-|i|)}.
\end{align}
Then since $p\mu > n$ we find that
\begin{align}
\sum_{-\infty}^{\infty}2^{\frac{p\mu-n}{p}(1-|i|)} &= 2^{\frac{p\mu-n}{p}+1}+2 \sum_{i=1}^{\infty}2^{\frac{p\mu-n}{p}(1-i)}
\\&= 2^{\frac{p\mu-n}{p}+1}+2 \sum_{i=0}^{\infty}\left(2^{\frac{n-p\mu}{p}}\right)^i=2^{\frac{p\mu-n}{p}+1}+\frac{2}{1-2^{\frac{n-p\mu}{p}}},
\end{align}
and hence 
\begin{align}
|u(x)-u(y)|&\le C(p,\mu, C_V,C_P) d_g(x,y)^{\frac{p\mu-n}{p}}\left( \int_{M} |\nabla u|^p dV_g \right)^{\frac{1}{p}}.\label{HolderQuotientBound}
\end{align}

Now, let $y \in B_1(x)$, and integrate both sides of $|u(x)| \le |u(y)|+|u(x)-u(y)|$ to obtain
\begin{align}
   \int_{B_1(x)} |u(x)|dV_g(y) \le \int_{B_1(x)}|u(y)|dV_g(y)+\int_{B_1(x)}|u(x)-u(y)|dV_g(y),
\end{align}
or, by a slight rearrangement,
\begin{align}
    |u(x)| \le \dashint_{B_1(x)}|u|dV_g+\dashint_{B_1(x)}|u(x)-u(y)|dV_g(y).
\end{align}
Apply \eqref{HolderQuotientBound}, \eqref{LowerVolumeBound}, and the fact that $d_g(x,y) \le 1$, to find that
\begin{align}
    |u(x)| &\le C_V\int_{B_1(x)}|u|dV_g
    \\&\quad+C(p,\mu,C_V,C_P)\dashint_{B_1(x)} d_g(x,y)^{\frac{p\mu-n}{p}}\left( \int_{M} |\nabla u|^p dV_g \right)^{\frac{1}{p}}dV_g(y)
    \\&\le C_V\int_{M}|u|dV_g+C(p,\mu, C_V,C_P)\left( \int_{M} |\nabla u|^p dV_g \right)^{\frac{1}{p}},
\end{align}
which, together with \eqref{HolderQuotientBound}, implies the desired H\"{o}lder bound.
\end{proof}
\subsection{Application}
Fix $b,\Lambda, \kappa,\bar{m}>0$ and $\tau>\tfrac12$. Using the above work, we can now show that members of the family $\mathcal{M}(b,\tau,\bar{m},\Lambda,\kappa)$ have well behaved Sobolev spaces. We begin by demonstrating that metric balls are well behaved.
\begin{prop}\label{prop:upper-lower-volume-bounds}
     Given $R>0$ there exists a constants $C(\Lambda)$ and $C(\kappa,R)$ such that for any member $(M,g)$ in $\mathcal{M}(b,\tau,\bar{m},\Lambda,\kappa)$ we have, for all $x\in M$ and all $r\leq R$, the following volume bounds
     \begin{equation}
         C(\Lambda)r^{3}\leq |B_r(x)|\leq C(\kappa,R)r^3.
     \end{equation}
\end{prop}
\begin{proof}
    We observe that the upper bound on $|B_r(x)|$ follows directly from the control we have on $Rm$ applied to Corollary \ref{volume-growth-above}.
    
    Let us now concentrate on the lower bound on volume growth. From our assumption on asymptotic flatness, it follows that for any $x$ in $M$ we have
    \begin{equation}
        \lim_{r\rightarrow\infty}|B_r(x)|=\infty,
    \end{equation}
    and any compact subset is eventually contained in $B_r(x)$ for some $r$ large enough. Thus, we may apply Proposition \ref{NeumanToDirichletIsoperimetricEst} together with our assumption on the Neumann isoperimetric constant of metric balls to conclude that any compact domain in $M$ has Dirichlet isoperimetric constant uniformly bounded away from zero, depending on $\Lambda>0$. At this point, we may apply Lemma \ref{DirichletIsoToLowerVolumeGrowth} to get the desired result.
\end{proof}

We can now state the form of Morrey's inequality which is valid uniformly for the family $\mathcal{M}(b,\tau,\bar{m},\Lambda,\kappa)$. It is obtained by gathering together all the information we established thus far. 
\begin{thm}\label{thm:Morrey's-inequality-balls}
    For all $(M,g)$ in $\mathcal{M}(b,\tau,\bar{m},\Lambda,\kappa)$, points $x_{0}$ in $M$, and functions $f\in W^{1,p}(M,g)$ with $p>3$ and $s>0$, we have 
    \begin{equation}
        \|f\|_{C^{0,\frac{p-3}{p}}_{R}\left(B_{s}(x_{0})\right)}\leq C(b,\tau,\Lambda,\kappa,R,p)\left(\|f\|_{L^{1}\left(B_{s+R}(x_{0})\right)}+\|\nabla f\|_{L^{p}\left(B_{s+R}(x_{0})\right)}\right).
    \end{equation}
\end{thm}
\begin{proof}
    The proof will be complete when we show that every member of the family $\mathcal{M}(b,\tau,\bar{m},\Lambda,\kappa)$ satisfies the conditions of Theorem \ref{SobolevToMorrey}, since then we can follow the proof of Theorem \ref{SobolevToMorrey} word for word with $M=B_{s+R}(x_0)$. Let us recall that we need to establish that metric balls around points in $B_{s}(x_{0})$ have upper and lower volume bounds, and uniformly satisfy a $(1,p)$ Poincar{\'e} inequality. 
    
    It follows from Proposition \ref{prop:upper-lower-volume-bounds} that $B_r(x)$ satisfies the required volume bounds for $r\leq R$ and $x$ in $B_{s}(x_{0})$. In addition, it follows from Lemma \ref{Poincare-from-uniform-SN} that $B_r(x)$ has the desired $(1,p)$ Poincar{\'e} inequality with $\mu=1$ for $r\leq R$ and $x\in B_{s}(x_0)$. Thus, we get the result by following the proof of Theorem \ref{SobolevToMorrey}.
\end{proof}
%%%%%%%%%%%%%%%%%%%%%%%%%%%%%%%%%%%%%%%%%%%%%%%%%%%%%%

\section{Asymptotic Estimates For the Harmonic Functions} \label{sec-Asymptotic}

In this section, we show how ADM mass and the condition of uniform asymptotic flatness strongly control the behavior of asymptotically linear harmonic functions within the asymptotically flat region. Let us first recall some information about asymptotically linear harmonic functions.

For a complete asymptotically flat $3$-manifold $(M,g)$, let $x^{i}$ for $i=1,2,3$ be the components of the asymptotic coordinate chart. Then we call $u^i$ the $i^{\text{th}}$ asymptotically linear harmonic function if $u^{i}$ satisfies
\begin{equation}
    \begin{split}
        &\Delta_{g}u^{i}=0
        \\ 
        &u^{i}-x^{i}\in C^{2,\gamma}_{1-\tau}.
    \end{split}
\end{equation}
In particular, the gradients of $u^i$ are asymptotically constant in the sense that
\begin{equation}\label{e:harmonicasym1}
    |\nabla u^{i}-\partial_{x^{i}}|\in C^{1,\gamma}_{-\tau}.
\end{equation}
{\bf{Convention:}} Given asymptotically flat $(M,g)$ and $r>0$, we denote by $M_r$ the bounded region of $M$ lying inside the coordinate sphere of radius $r$, denoted $\mathcal{S}_r$. Secondly, if $\ell$ is a linear map from $\mathbb{R}^3\to\mathbb{R}$ with $||d\ell||=1$, we will call any solution to $\Delta u=0$ with $|u-\ell(x^1,x^2,x^3)|\in C^{2,\gamma}_{1-\tau}(M)$, for some $\gamma\in(0,1)$, an {\em{asymptotically linear harmonic function}}, noting that we always require the linear function to have unit norm. Also, throughout this section, we will fix parameters $b,\Lambda, \kappa,\bar{m}>0$ and $\tau>\tfrac12$. 

It was observed in \cite{KKL-2021} that the asymptotics \eqref{e:harmonicasym1} are uniform in the following sense.
\begin{prop}\label{asymptotic-gradient-convergence}
    Let $(M,g)$ be an orientable complete $3$-dimensional $(b,\tau,\bar{m})$ asymptotically flat manifold. Assume that $H_2(M;\mathbb{Z})$ contains no spherical classes and that $R_g\geq0$. For a sufficiently large $r_0$ (depending on $b$ and $\tau$), there is a constant $C(r_0,b,\tau,\bar m)$ so that any asymptotically linear functions $\{u^i\}_{i=1}^3$ satisfy
    \begin{equation}\label{e:weighteddecay}
        ||\nabla u^i-\partial_{x^i}||_{C^{1,\alpha}_{-\tau}(M\setminus M_{r_0})}\leq C(r_0,b,\tau,\bar m)
    \end{equation}
   and in particular,
    \begin{equation}\label{e:asymptotic-gradient}
        \left|g\left(\nabla u^{i},\nabla u^{j}\right)-\delta^{ij}\right|\leq C(r_0,b,\tau,\bar{m})|x|^{-\tau}
    \end{equation}
    holds on $M\setminus M_{r_0}$. Moreover, there is a particular collection of asymptotically linear functions $\{\hat{u}^i\}_{i=1}^3$, independent of $r_0$, which satisfy 
    \begin{equation}\label{e:c0control}
        \sup_{M_{r}}|\hat{u}^i|\leq C(r_0,b,\tau,\bar{m})r
    \end{equation}
    for all $r>r_0$.
\end{prop}
\begin{proof}
    Only the final statement requires proof since the other statements are contained in \cite[Lemma 3.2]{KKL-2021}. Fix the choice of $\{\hat{u}^i\}_{i=1}^3$ by requiring that the average value of $\hat{u}^i$ over the annulus $M_{r_0}\setminus M_1$ vanishes. According to \cite[Proposition 3.1]{KKL-2021}, using a Sobolev inequality and the mass formula, we have $\sup_{M_{r_0}}|\hat{u}^i|< C_1(r_0,b,\tau,\bar{m})$. For a larger radius $r>r_0$, one can estimate $\hat{u}^i$ at a point $x\in (M_r\setminus M_{r_0})$ by integrating the gradient estimate \eqref{e:weighteddecay} along a radial curve connecting $x$ to $M_{r_0}$. The fundamental theorem of calculus then shows that $|\hat{u}^i(x)-C_1(r_0,b,\tau,\bar{m})|<C_2(r_0,b,\tau,\bar{m})r$ and the result follows.
\end{proof}

Next, we leverage the uniform weighted control \eqref{e:weighteddecay} to show that the difference between $g$ and the flat metric on $\mathbb{R}^3$ is small in the average sense within the asymptotic region.

\begin{prop}\label{p:asymptoticannulusdecay}
     Let $(M,g)$ be an oriented complete $3$-dimensional $(b,\tau,\bar{m})$ asymptotically flat manifold with mass $m$. Assume that $H_2(M;\mathbb{Z})$ contains no spherical classes and that $R_g\geq0$. Given a sufficiently large $r_0=r_0(b,\tau)>0$, $r>r_0$, and an $\varepsilon>0$, there is a $\delta=\delta(r_0,b,\tau,\varepsilon)>0$ so that the following holds: If $m(M,g)\leq\delta$, then
    \begin{equation}\label{e:asymptoticannulusdecay}
        \int_{M_{r}\setminus M_{r_0}}\left|\left<\nabla u^{i},\nabla u^{j}\right>-\delta^{ij}\right|dV_g<\varepsilon,
    \end{equation}
    for any $i,j\in\{1,2,3\}$ where $\{u^i\}_{i=1}^3$ are the harmonic coordinates of $(M,g)$.
\end{prop}

\begin{proof}
    Let $r_0>0$ be the radius given in Proposition \ref{asymptotic-gradient-convergence}. By \eqref{e:weighteddecay}, 
    \begin{equation}
        |\nabla^{2}u^{i}|+|\nabla u^{i}|\leq C_1(b,\tau)
    \end{equation}
    holds on $M_{r_0}$. In light of the $|\nabla u^{i}|$'s boundedness, the mass formula \eqref{ADM-formula} implies
    \begin{equation}
        \int_{M\setminus M_{r_{0}}}|\nabla^{2}u^i|^2dV\leq C_2(b,\tau)m.
    \end{equation}

    The next step is to apply a Sobolev inequality on asymptotically flat manifolds due to Schoen-Yau \cite[Lemma 3.1]{SchoenYauPMT},
    \begin{align}\label{e:schoenyausobo}
        \int_{M\setminus M_{r_0}}(\langle \nabla u^i,\nabla u^j\rangle-\delta^{ij})^6dV&\leq C_3(b,\tau)\left(\int_{M\setminus M_{r_0}}|\nabla\langle \nabla u^i,\nabla u^j\rangle|^2 \right)^3\\
        {}&\leq C_4(b,\tau)\sup_{M\setminus M_{r_0}}|\nabla u^i|^2\left(\int_{M\setminus  M_{r_0}}|\nabla^2u^j|^2dV\right)^3
        \\{}
        &+C_4(b,\tau)\sup_{M\setminus M_{r_0}}|\nabla u^j|^2\left(\int_{M\setminus M_{r_0}}|\nabla^2u^i|^2dV\right)^3\\
        {}&\leq C_5(b,\tau)m^3.
    \end{align}
    Also notice that the uniform asymptotic flatness condition implies that $|M_{r}\setminus M_{r_{0}}|\leq C_{6}(b,\tau)r^{3}$. Combining this observation with H{\"o}lder's inequality and \eqref{e:schoenyausobo},
    \begin{align}
        \int_{M_{r}\setminus M_{r_0}}|\langle \nabla u^i,\nabla u^j\rangle-\delta^{ij}|dV&\leq |M_{r}\setminus M_{r_0}|^{5/6}\left(\int_{M_{r_0}\setminus M_r}|\langle \nabla u^i,\nabla u^j\rangle-\delta^{ij}|^6dV\right)^{1/6}\\
        {}&\leq C_7(b,\tau)r^{5/2}\sqrt{m}
    \end{align}
    and the result follows.
    
%    Furthermore, from the uniform asymptotic flatness condition, it follows that $|M_{r}\setminus M_{r_{0}}|\leq C_{4}r^{3}$ for a $C_{4}=C_4(b,\tau)$.    Consequently, $\delta=\frac{\varepsilon}{2\left|B_{r}\setminus B_{r_{0}}\right|}$ is bounded away from $0$.. Let $R_{1}$ be the constant we get in Prop if we choose $\varepsilon=\delta$ there. We may as well assume that $R_{1}\geq 2 R$, to simplify the proof. For any $x$ in $B_{R}$, let $y$ in $S_{R_{1}}$ be the point closet to it, with respect to the Euclidean distance (this is not essential, it just the argument a little easier). Then, there is a symmetric family of curves connecting $y$ to $x$. Applying the coarea formula to this family of curves shows that there exists at least one curve connectin $y$ to $x$, say $\gamma$, such that
%    \begin{equation}
%        \int_{\gamma}|\nabla^{2}u^{i}_{j}|\leq K(q)\|\nabla^{2}u^{i}_{k}\|_{L^{q}}.
%    \end{equation}
%    In particular, we have that
%    \begin{equation}
%        \left|\left<\nabla u^{i}_{k},\nabla u^{j}_{k}\right>_{g}-\delta^{ij}\right|\leq\delta+\int_{\gamma}|\nabla^{2}u^{i}_{k}||\nabla u^{j}_{k}|+|\nabla^2u^{j}_{k}||\nabla u^{i}_{k}|.
%    \end{equation}
%    Since $|\nabla u^{i}_{k}|\leq C$, it follows that we have
%    \begin{equation}
%        \left|\left<\nabla u^{i}_{k},\nabla u^{j}_{k}\right>_{g}-\delta^{ij}\right|\leq \delta +CK(q)\left(\|\nabla^2u^{i}_{k}\|_{L^{q}}+\|\nabla^2u^{j}_{k}\|\right).
%    \end{equation}
%    As observed above, the second term on the right hand side is bounded by the mass.
\end{proof}

\begin{thm}\label{C3AsymptoticEstimate}
    Let $(M,g)$ be a $(b,\tau,\bar{m})$ asymptotically flat Riemannian manifold. If $u$ is a asymptotically linear harmonic function, then there is a $r_0(b,\tau)$ so that the following holds on $M \setminus M_{r_0}$
    \begin{align}
        |\nabla u|+|\nabla^2 u|+|\nabla^3 u|\le C(b,\tau).
    \end{align}
  
\end{thm}
\begin{proof}
   Throughout this proof we will work in $(M,g)$'s asymptotically flat coordinate chart and all partial derivatives will be in these coordinates. Using Proposition \ref{asymptotic-gradient-convergence}, we know that we can choose an $r_0$ large enough so that the first and second derivatives satisfy $|\partial u| \le C_1(b,\tau)$ and $|\partial\partial u|\leq C_2(b,\tau)$ on $M \setminus M_{r_0}$. It suffices, therefore, to estimate the third derivatives of $u$ and relate the partial derivatives to covariant derivatives with respect to the connection induced by $g$. 
   Expanding the Hessian of $u$ in asymptotically flat coordinates,
   \begin{align}
       \nabla^2u=g^{lm}\left(\partial_l\partial_m-\Gamma_{lm}^k\partial_k \right)u. \label{HessianInCoords}
   \end{align}
   According to the uniform asymptotics, the terms in \eqref{HessianInCoords} can be bounded in $M_{r_0}$ to find $|\nabla^2 u| \le C_3(b,\tau)$.
   
   To estimate the third derivatives of $u$, we apply the B{\"{o}}chner formula to find
   \begin{align}
       \Delta \nabla u=g^{lm}\left(\partial_l\partial_m -\Gamma_{lm}^k \partial_k \right) \nabla u=Rc(\nabla u,\cdot)=:\mathbf{f},
   \end{align}
   which we consider as a PDE for $\nabla u$. Again using the uniform asymptotics, $|\mathbf{f}|, |g^{lm}|,|\Gamma_{lm}^k| \le C_4(b,\tau)$ on $M \setminus M_{r_0}$. Applying elliptic  $L^p$ estimates \cite[Theorem 9.11]{GT}, for every $p >1$
   \begin{align}
       \|\partial_i u\|_{W^{2,p}(M\setminus M_{r_0},\delta)} \le C_5(b,\tau), \quad 1 \le i \le 3. \label{W2pCoordEst}
   \end{align}
   Since \eqref{W2pCoordEst} holds for all $p>1$, the third partial derivatives $|\partial\partial\nabla u| \le C_6(b,\tau)$
    and again since
   \begin{align}
       \nabla^2\nabla_i u=g^{lm}\left(\partial_l\partial_m-\Gamma_{lm}^k\partial_k \right)\nabla_i u, \label{HessianInCoords2}
   \end{align}
  we find $|\nabla^3 u| \le C_8(b,\tau)$, as claimed.
\end{proof}

%\begin{lem}\label{l:spherelocation}
%Given $C,\tau>0$, there is a $\rho_1=\rho_1(C,\tau)$ and a $\rho_2=\rho_2(C,\tau)$ satisfying the following: For any connected complete $(C,\tau)$-asymptotically flat manifold $(M^n,g)$, there is a point $x\in M$ and a radius $r_0$ so that 
%\begin{equation}
%    \partial B_{r_0+r}(x)\subset M_{r+r_2}\setminus M_{r-r_2}
%\end{equation}
%and
%\begin{equation}
%    |\partial B_{r_0+r}(x)|\leq 10(1+C)^{n-1}(r+r_2)^{n-1}
%\end{equation}
%for all $r\geq \rho_1$.
%\end{lem}

%\begin{proof}
%    Let $(M,g)$ be a $(C,\tau)$-asymptotically flat manifold. There is a radius, which we will fix and denote by $r^*$, depending only on $C,\tau$ so that $|g-\delta|< 1/100$ outside $M_{r^*}$. Now choose a point $x\in M_{r^*}$ of maximal distance $D$ from $\partial M_{r^*}$. Since the diameter of $\partial M_{r^*}$ is no larger than $\frac{101\pi r^*}{100}$, the component of $\partial B_{D+\frac{101\pi}{100}}(x)$ lying in the asymptotic region is contained in the coordinate annulus $M_{r^*+2\pi}\setminus M_{r^*}$.
    
%    Now let $r>0$ and let's focus our attention on the balls $B_{2D+2\pi}(x)$. Since all points in $M_{r^*}$ are within distance $D$ from $\partial M_{r^*}$, the only component of $\partial B_{2D+2\pi}(x)$ is the one lying in the asymptotic region. 
    
%\end{proof}

We conclude this section by showing that there is a strong relationship between the regions $M_r$ and metric balls $B(x,r)$ for manifolds in the class $\mathcal{M}(b,\tau,\bar{m},\Lambda,\kappa)$. This information is important since the regions $M_r$ have well behaved boundary, but no inherent geometric significance, and the balls $B(x,r)$ are geometrically important, but could have wild boundaries. 
Before accomplishing this, we show that a uniform volume bound on the region within a given coordinate sphere.

\begin{lem}\label{DirichletToVolumeBound}
Let $(M,g)$ belong to $\mathcal{M}\left(b,\tau,\bar{m},\Lambda,\kappa\right)$. Then, there exists a constant $K=K(r_0,b,\tau,\Lambda)$ such that, for any $r>r_0$, we have $|M_{r}|\leq K r^{3}.$
\end{lem}
\begin{proof}
    By the definition of $M_r$ and the $(b,\tau)$ asymptotic flatness of $M$ on $M_{r_0}$, we know that
    \begin{equation}
        |\mathcal{S}_{r}| \leq C_1(r_0,b,\tau)r^2.
    \end{equation}
     Further, $\mathcal{S}_r$ is a valid competitor for the isoperimetric ratio, and so it follows from the fact that $M$ is $(\Lambda,\frac{3}{2})$ Neumann-isoperimetrically bounded and Lemma \ref{NeumanToDirichletIsoperimetricEst} that
    \begin{align}
        |M_r| \le C_2(r_0,b,\tau,\Lambda)r^{3}.
    \end{align}
\end{proof}

The following lemma allows us to simultaneously take advantage of the good boundary behavior of the regions $M_{r}$ and the good geometric behavior of the metric balls $B(x,r)$.
\begin{figure}[htp]
    \centering
    \begin{subfigure}[b]{0.3\textwidth}
        \centering
        \includegraphics[width=4cm]{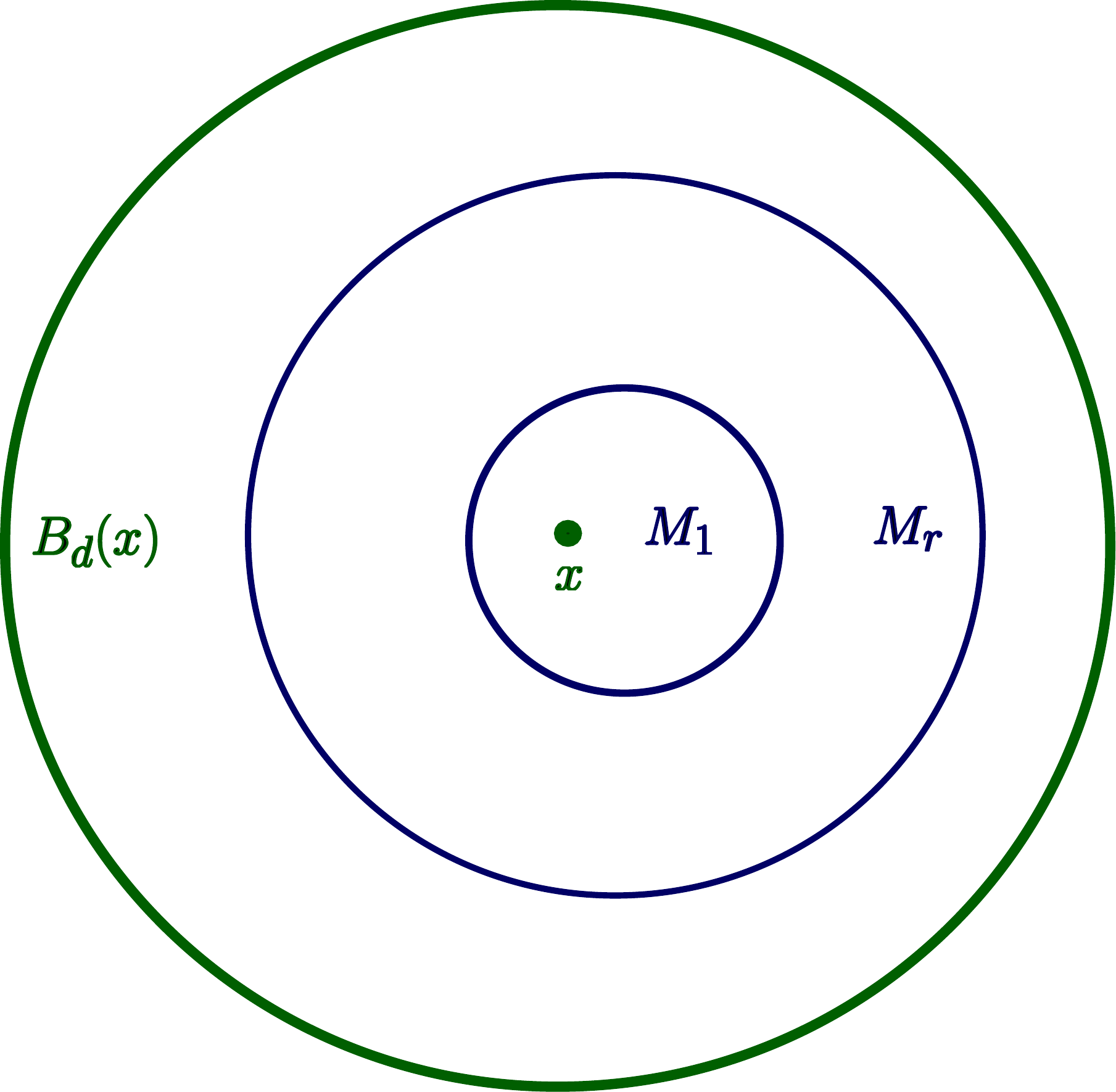}
        \caption{$M_{r}$ is contained in $B_{d}(x)$}
    \end{subfigure}
    \begin{subfigure}[b]{0.3\textwidth}
        \centering
        \includegraphics[width=4cm]{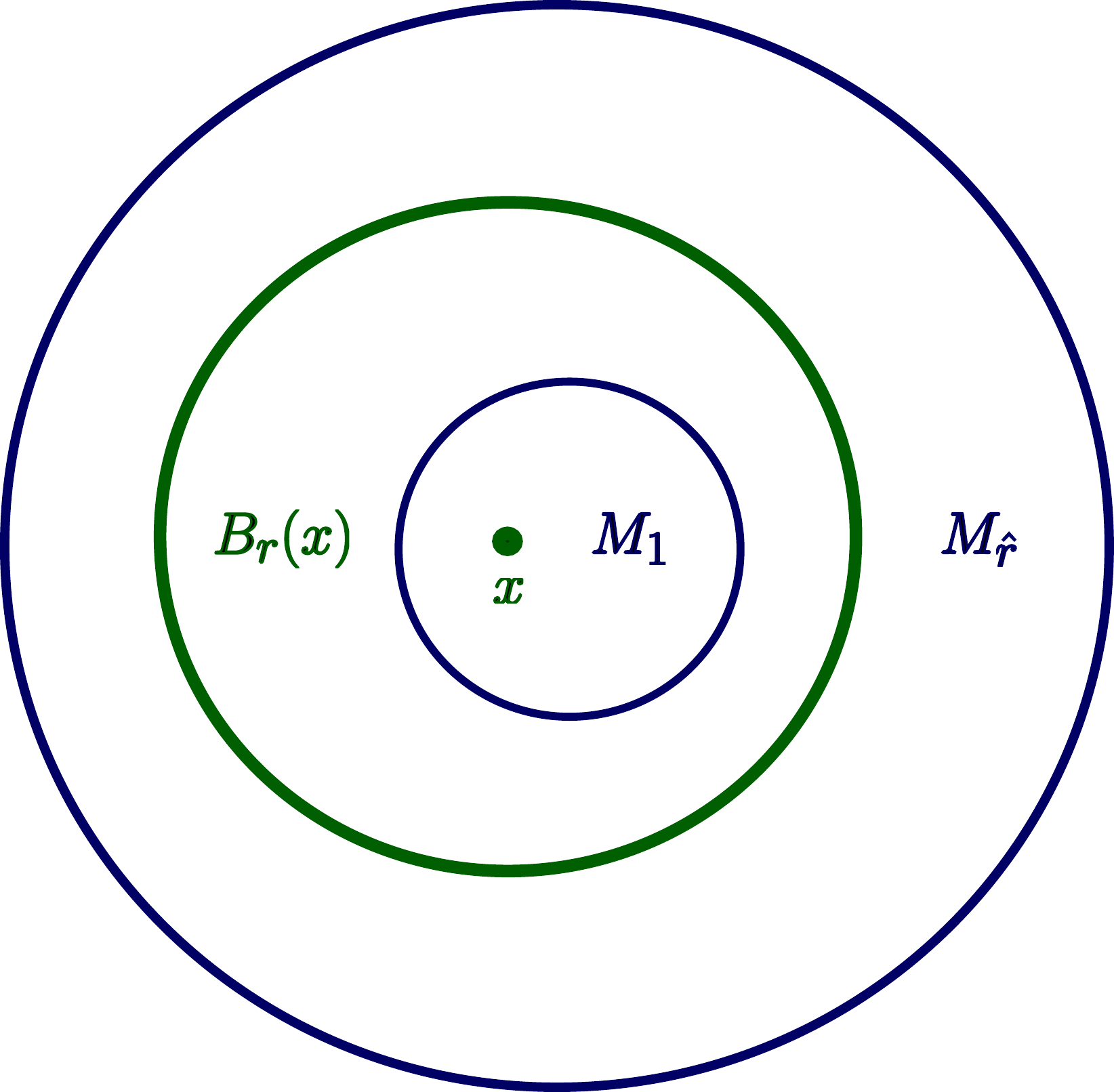}
        \caption{$B_r(x)$ is contained $M_{\hat{r}}$}
    \end{subfigure}
\end{figure}
\begin{lem}\label{AsymDiamEst}
Given $b,\tau,\Lambda,r$, there is $d=d(b,\tau,\Lambda,r)$ and $V=V(b,\tau,\Lambda,r)$ so that the following holds: If $(M,g)\in\mathcal{M}(b,\tau,\bar{m},\Lambda,\kappa)$, then
$\mathrm{diam}(M_r)\leq d$ and $|M_r|\leq V$. 

Moreover, there is an $\hat{r}=\hat{r}(b,\tau,r)$ so that the following holds: If $(M,g)\in\mathcal{M}(b,\tau,\bar{m},\Lambda,\kappa)$ and $x\in M_1$, then $B_r(x)\subset M_{\hat{r}}$.
\end{lem}
\begin{proof}
    From Lemma \ref{DirichletToVolumeBound} it follows that $|M_{r}|$ is bounded, say by $V=V(b,\tau,\kappa,r)$. Let $x$ and $y$ be any two points in $M_{r}$. By Lemma \ref{DirichletIsoToLowerVolumeGrowth}, we know that $|B_r(x)|$ and $|B_r(y)|$ have volume bounded below by $C(\Lambda) r^{3}$. Let $S=\left(\frac{V+1}{C(\Lambda)}\right)^{\frac{1}{n}}$. Then, by looking at the relevant volumes, we can conclude that
    \begin{equation}
        \begin{split}
            B_{S}(x)\bigcap \partial M_{r}\ne \emptyset,\qquad
            B_{S}(y)\bigcap \partial M_{r}\ne \emptyset.
        \end{split}
    \end{equation}
    Since $\partial M_{r}$ is in the asymptotically flat region, it follows that $\Diam(\partial M_{r})$ is uniformly bounded above, by $C=C(b,\tau,\kappa,r)$. Putting the above together, we see that we have
    \begin{equation}
        \Diam(M_{r})\leq 2S+C=:d,
    \end{equation}
    proving the first statement of Lemma \ref{AsymDiamEst}.
    
    For the second statement, let $x\in M_1$ and consider the metric ball $B_r(x)$. Using the uniform asymptotics, we can find a large enough $\hat{r}=\hat{r}(b,\tau,r)$ so that any path from $\partial M_1$ to $\partial M_{\hat{r}}$ has length least $r$. It follows that  $B_r(x)\subset M_{\hat{r}}$.
\end{proof}

% This leads to a useful corollary, which bounds the volume growth of balls above.
% \begin{cor}\label{cor:volume-growth-bound-balls}
%     Let $(M,g)$ belong to $\mathcal{M}(b,\tau,\bar{m},\Lambda,\frac{3}{2},\kappa,p)$ for some $p>\frac{3}{2}$. Then, there exists a constant $K(b,\tau,\Lambda,\kappa,p,r_0)$ such that for any $x$ in $M_{r_0}$ and $r\leq\Diam\left(M_{r_0}\right)$ we have
%     \begin{equation}
%         \Vol\left(B_{r}(x_0)\right)\leq K(b,\tau,\Lambda,\kappa,p,r_0)r^{3}.
%     \end{equation}
% \end{cor}
% \begin{proof}
%     From Lemma \ref{AsymDiamEst}, it follows that
%     \begin{equation}
%         \Diam\left(M_{r_0}\right)\leq C(b,\tau,\kappa,\Lambda,r_{0}).
%     \end{equation}
%     We may now plug $C(b,\tau,\kappa,\Lambda,r_{0})$ into Corollary \ref{volume-growth-above}, and observe that $r\leq C(b,\tau,\kappa,\Lambda,r_{0})$.
% \end{proof}

%%%%%%%%%%%%%%%%%%%%%%%%%%%%%%%%%%%%%%%%%%%%%%%%%%%%%%
\section{Integral Estimates For Harmonic Functions}\label{sec:int-est-harmonic-func} \label{sec-IntegralEstimates}

In this section we leave the asymptotically flat setting and work with general and fixed oriented complete $n$-dimensional Riemannian manifold $(M,g)$, and study harmonic functions $u:M\rightarrow \R$, which is to say, $\Delta u = 0$. We would like to develop integral estimates for $u$. Let $\Omega \subset M$ be bounded region with smooth boundary $\partial \Omega$ and outward pointing unit normal vector $\nu$. Our goal in this section is to obtain Sobolev type estimates for harmonic functions, which will ultimately lead to H\"{o}lder control.
We begin with an elementary observation which leads to $L^2$ control on the gradient of harmonic functions.
\begin{prop}\label{FirstOrderEstimate}
Let $(M,g)$ be an oriented complete $n$-dimensional Riemannian manifold, $u$ a harmonic function, and $\Omega \subset M$ a bounded domain with smooth boundary $\partial \Omega$ and outward pointing unit normal vector $\nu$. Then
\begin{align}
\int_{\Omega} | \nabla u|^2 dV_g
&=\int_{\partial \Omega}u\frac{\partial u}{\partial \nu} dA_g.
\end{align}
\end{prop}
\begin{proof}
We can integrate by parts $-u \Delta u=0$ to find
\begin{align}
    \int_{\Omega}|\nabla u|^2dV_g&= \int_{\partial \Omega}u\frac{\partial u}{\partial \nu}dA_g.
\end{align}
\end{proof}

Now we obtain an $L^2$ estimate on the hessian of harmonic functions.
\begin{thm}\label{SecondOrderEstimate}
Let $(M,g)$ be an oriented complete $n$-dimensional Riemannian manifold, $u$ a harmonic function, and $\Omega \subset M$ a bounded domain with smooth boundary $\partial \Omega$ and outward pointing unit normal vector $\nu$. Then
\begin{align}
\int_{\Omega} |\nabla^2 u|^2 dV_g
&\le\int_{\partial \Omega}|\nabla^2 u||\nabla u| dA_g +\left(\int_{\Omega} |Rc|^pdV_g\right)^{1/p} \left(\int_{\Omega}|\nabla u|^{\frac{2p}{p-1}}dV_g \right)^{\frac{p-1}{p}}.
\end{align}

\end{thm}
\begin{proof}
Now we calculate
\begin{align}
0&=\int_{\Omega} (\Delta u)^2 dV_g
\\&= \int_{\Omega}g^{ij} g^{pq} \nabla_i\nabla_ju \nabla_p \nabla_qu dV_g
\\&= -\int_{\Omega} g^{ij} g^{pq} \nabla_ju \nabla_i\nabla_p \nabla_qu dV_g+\int_{\partial \Omega} g^{ij} g^{pq}\nabla_ju \nabla_p\nabla_qu \nu_i dA_g
\\&= -\int_{\Omega} g^{ij} g^{pq} \nabla_ju \nabla_p\nabla_i \nabla_qu -g^{ij} g^{pq} \nabla_juR_{ipqk} \nabla_k u dV_g+\int_{\partial \Omega} \frac{\partial u}{\partial \nu} \Delta u dA_g
\\&= \int_{\Omega} g^{ij} g^{pq} \nabla_p\nabla_ju \nabla_i \nabla_qu +Rc(\nabla u,\nabla u)dV_g - \int_{\partial \Omega}g^{ij} g^{pq} \nabla_j u \nabla _i \nabla _q u \nu_p dA_g 
\\&= \int_{\Omega} g^{ij} g^{pq} \nabla_p\nabla_ju \nabla_q \nabla_iu +Rc(\nabla u,\nabla u)dV_g - \int_{\partial \Omega}\nabla \nabla u(\nabla u, \nu) dA_g 
\\&= \int_{\Omega} |\nabla^2 u|^2 +Rc(\nabla u,\nabla u)dV_g - \int_{\partial \Omega}\nabla \nabla u(\nabla u, \nu) dA_g,
\end{align}
and so by rearranging we find
\begin{align}
\int_{\Omega} |\nabla^2 u|^2 dV_g=\int_{\partial \Omega}\nabla \nabla u(\nabla u, \nu) dA_g -\int_{\Omega} Rc(\nabla u,\nabla u)dV_g.
\end{align}
If we take the absolute value of both sides, use Cauchy-Schwarz, and then H\"{o}lder's inequality we find
\begin{align}
\int_{\Omega} |\nabla^2 u|^2 dV_g&\le\int_{\partial \Omega}|\nabla^2 u||\nabla u| dA_g +\int_{\Omega} |Rc||\nabla u|^2dV_g
\\&\le\int_{\partial \Omega}|\nabla^2 u||\nabla u| dA_g +\left(\int_{\Omega} |Rc|^pdV_g\right)^{1/p} \left(\int_{\Omega}|\nabla u|^{\frac{2p}{p-1}}dV_g \right)^{\frac{p-1}{p}}.
\end{align}
\end{proof}
Now we would like to estimate the integral of third derivatives of harmonic functions in a similar way. Below and throughout, given two tensors $T_1$ and $T_2$, we will use $T_1*T_2$ as a stand in for any linear function of $T_1$ and $T_2$.

\begin{thm}\label{FirstThirdOrderEstimate}
Let $(M,g)$ be an oriented complete $n$-dimensional Riemannian manifold, $u$ a harmonic function, and $\Omega \subset M$ a bounded domain with smooth boundary $\partial \Omega$ and outward pointing unit normal vector $\nu$. Then
\begin{align}
\begin{split}
\int_{\Omega} |\nabla^3 u|^2dV_g&=\int_{\Omega} 2Rc(\nabla u, \Delta \nabla u) -Rc^2(\nabla u, \nabla u)-Rc*\nabla\nabla u*\nabla\nabla u dV_g
\\&+\int_{\Omega}Rm*\nabla\nabla\nabla u* \nabla u  -Rm*\nabla\nabla u*\nabla\nabla udV_g
\\&+\int_{\partial \Omega} \nabla\nabla\nabla u(\nabla\nabla u,\nu)-\nabla \nabla u(\Delta \nabla u, \nu)dA_g,
\end{split}
\end{align}
where $Rm$ denotes the Riemann curvature tensor of $(M,g)$.
\end{thm}
\begin{proof}
Now we notice that since $u$ is harmonic we can also deduce that $\nabla \Delta u=0$ and calculate
\begin{align}
0&= \int_{\Omega} |\nabla \Delta u|^2 dV_g
\\&= \int_{\Omega} g^{ij} g^{pq} g^{lm} \nabla_i \nabla_p \nabla_q u \nabla_j \nabla_l \nabla_mu dV_g
\\&= \int_{\Omega} g^{ij} g^{pq} g^{lm} \left(\nabla_p \nabla_i \nabla_q u -g^{rs}R_{ipqr}\nabla_s u\right) \left(\nabla_l \nabla_j \nabla_mu -g^{vw}R_{jlmv}\nabla_wu\right)dV_g
\\&= \int_{\Omega} g^{ij} g^{pq} g^{lm} \left(\nabla_p \nabla_i \nabla_q u\nabla_l \nabla_j \nabla_mu -2g^{rs}R_{ipqr}\nabla_s u\nabla_l \nabla_j \nabla_mu +g^{vw}R_{jlmv}\nabla_wug^{rs}R_{ipqr}\nabla_s u\right)dV_g
\\&= \int_{\Omega} g^{ij} g^{pq} g^{lm} \nabla_p \nabla_q \nabla_i u\nabla_l \nabla_m \nabla_ju -2g^{ij} g^{lm}g^{rs}R_{ir}\nabla_s u\nabla_l\nabla_m \nabla_ju +Rc^2(\nabla u, \nabla u)dV_g
\\&= \int_{\Omega} g^{ij} g^{pq} g^{lm} \nabla_p \nabla_q \nabla_i u\nabla_l \nabla_m \nabla_ju -2Rc(\nabla u, \Delta \nabla u) +Rc^2(\nabla u, \nabla u)dV_g
\\&= I+II+III.
\end{align}
For $III$ we note that
\begin{align}
Rc^2(\nabla u, \nabla u)=g^{ij}g^{pq}g^{rs} R_{ip}R_{jr} \nabla_q u \nabla_s u,
\end{align}
and for $II$ we note that
\begin{align}
Rc(\nabla u, \Delta \nabla u)=g^{ij} g^{lm}g^{rs}R_{ir}\nabla_s u\nabla_l\nabla_m \nabla_ju.
\end{align}
Now we continue the calculation by focusing on $I$
\begin{align}
I&= -\int_{\Omega} g^{ij} g^{pq} g^{lm}  \nabla_q \nabla_i u\nabla_p\nabla_l \nabla_m \nabla_ju dV_g+\int_{\partial \Omega} g^{ij} g^{pq} g^{lm}\nabla_q \nabla_i u\nabla_l \nabla_m \nabla_ju \nu_p dA_g
\\&= -\int_{\Omega} g^{ij} g^{pq} g^{lm}  \nabla_q \nabla_i u\nabla_p\nabla_l \nabla_m \nabla_ju dV_g+\int_{\partial \Omega} \nabla \nabla u(\Delta \nabla u, \nu) dA_g=A+B.
\end{align}
Now we continue with the first term
\begin{align}
A&= -\int_{\Omega} g^{ij} g^{pq} g^{lm}  \nabla_q \nabla_i u \left(\nabla_l\nabla_p \nabla_m \nabla_ju - g^{rs}R_{plmr}\nabla_s\nabla_ju - g^{rs}R_{pljr}\nabla_m\nabla_su \right)dV_g
\\&= \int_{\Omega} -g^{ij} g^{pq} g^{lm}  \nabla_q \nabla_i u \nabla_l\nabla_p \nabla_m \nabla_ju + Rc*\nabla\nabla u*\nabla\nabla u + Rm*\nabla\nabla u*\nabla\nabla u dV_g
\\&=C+D+E.
\end{align}
For this calculation we are using
\begin{align}
Rc*\nabla\nabla u*\nabla\nabla u&= g^{ij}g^{pq}g^{rs}R_{pr}\nabla_s\nabla_j u \nabla_q \nabla_iu,
\\Rm*\nabla\nabla u*\nabla\nabla u&=g^{ij}g^{lm}g^{rs}R_{pljr}\nabla_m\nabla_s u \nabla_q \nabla_iu.
\end{align}
Now we continue with the term $C$ to find
\begin{align}
C&=\int_{\Omega} g^{ij} g^{pq} g^{lm}  \nabla_l\nabla_q \nabla_i u \nabla_p \nabla_m \nabla_ju dV_g -\int_{\partial \Omega} g^{ij} g^{pq} g^{lm}\nabla_q \nabla_i u  \nabla_p \nabla_m \nabla_ju \nu_ldA_g
\\&=\int_{\Omega} g^{ij} g^{pq} g^{lm}  \nabla_l\nabla_q \nabla_i u \nabla_p \nabla_m \nabla_ju dV_g -\int_{\partial \Omega} g^{ij} g^{pq} g^{lm}\nabla_q \nabla_i u  \nabla_p \nabla_j \nabla_mu \nu_ldA_g
\\&=\int_{\Omega} g^{ij} g^{pq} g^{lm}  \nabla_l\nabla_q \nabla_i u \left(\nabla_m \nabla_p \nabla_ju -g^{rs}R_{pmjr}\nabla_su\right)dV_g -\int_{\partial \Omega} \nabla\nabla\nabla u(\nabla\nabla u,\nu)dA_g
\\&=\int_{\Omega} |\nabla^3 u|^2 -g^{ij} g^{pq} g^{lm} g^{rs}R_{pmjr}\nabla_l\nabla_q \nabla_i u\nabla_sudV_g -\int_{\partial \Omega} \nabla\nabla\nabla u(\nabla\nabla u,\nu)dA_g
\\&=\int_{\Omega} |\nabla^3 u|^2 -Rm*\nabla\nabla\nabla u* \nabla udV_g -\int_{\partial \Omega} \nabla\nabla\nabla u(\nabla\nabla u,\nu)dA_g.
\end{align}
In the calculation above we are using
\begin{align}
Rm*\nabla\nabla\nabla u* \nabla u=g^{ij} g^{pq} g^{lm} g^{rs}R_{pmjr}\nabla_l\nabla_q \nabla_i u\nabla_su.
\end{align}

Putting everything together we find
\begin{align}
\begin{split}
\int_{\Omega} |\nabla^3 u|^2dV_g&=\int_{\Omega} 2Rc(\nabla u, \Delta \nabla u) -Rc^2(\nabla u, \nabla u)-Rc*\nabla\nabla u*\nabla\nabla u dV_g
\\&+\int_{\Omega}Rm*\nabla\nabla\nabla u* \nabla u  -Rm*\nabla\nabla u*\nabla\nabla udV_g
\\&+\int_{\partial \Omega} \nabla\nabla\nabla u(\nabla\nabla u,\nu)-\nabla \nabla u(\Delta \nabla u, \nu)dA_g
\end{split}
\end{align}
\end{proof}

The terms on the right hand side of the estimate in Theorem \ref{FirstThirdOrderEstimate} involving a square of the Hessian of $u$ are concerning since we should only expect to have integral control on $|\nabla^2 u|^2$ and we will need to use H\"{o}lder's inequality. To this end we prove the following proposition which will be used to overcome this issue. The proof of the following proposition is inspired by the calculations in \cite{Strzelecki-06}. However, we only need the following weak conclusion, whereas the result obtained in \cite{Strzelecki-06} is much more refined. Luckily, this significantly simplifies the proof.

\begin{prop}\label{prop:HarmonicInterpolationIneq}
    Let $(M,g)$ be an oriented complete $n$-dimensional Riemannian manifold, $u$ a harmonic function, and $\Omega \subset M$ a bounded domain with smooth boundary $\partial \Omega$. If $f\in C^{3}(\Omega)$, then we have that
   \begin{equation}
       \|\nabla^2 f\|^{2}_{L^{3}}\le \max\left(1, \|\nabla^2 f\|^{2}_{L^{3}}\right)\leq \max\left(1,5\|\nabla f\|_{L^{6}}\|\nabla^3 f\|_{L^{2}}+\int_{\partial\Omega}\left|\nabla^2f\right|^2 |\nabla f|dA_g\right).
    \end{equation}
\end{prop}
\begin{proof}
    The proof is the result of a nice integration by parts. In particular, we have that
    \begin{equation}
        \int_{\Omega}|\nabla^2 f|^{3}dV_g=-\int_{\Omega}\left<\nabla f,\nabla^{*}(|\nabla^2f|\nabla^{2}f)\right>dV_g+\int_{\partial\Omega}\left|\nabla^2f\right|\nabla^{2}f\left(\nabla f,\nu\right)dA_g
    \end{equation}
    where $\nu$ is the unit outward normal to $\partial \Omega$.
    Taking absolute values on the right hand side, the first term can be estimated as follows:
    \begin{align}
            \left|\left<\nabla f,\nabla^{*}(|\nabla^2f|\nabla^{2}f)\right>\right|\leq|\nabla f|\left(4|\nabla^2f||\nabla^3f|+|\nabla^2f||\nabla^3f|\right) = 5|\nabla f| |\nabla^2 f||\nabla ^3 f|.
    \end{align}
    Hence we have
 \begin{align}
        \int_{\Omega}|\nabla^2 f|^{3}dV_g\le5\int_{\Omega}|\nabla f| |\nabla^2 f||\nabla ^3 f|dV_g+\int_{\partial\Omega}\left|\nabla^2f\right|^2 |\nabla f|dA_g.
    \end{align}
    We may now use the generalized H{\"o}lder's inequality with exponents $6$, $3$, and $2$, respectively to get
    \begin{equation}
        \int_{\Omega}|\nabla^2 f|^{3}dV_g\leq 5\|\nabla f\|_{L^{6}}\|\nabla ^2 f\|_{L^{3}}\|\nabla^3 f\|_{L^2}+\int_{\partial\Omega}\left|\nabla^2f\right|^2 |\nabla f|dA_g.
    \end{equation}
    Dividing out both sides of the above by $\|\nabla ^2 f\|_{L^{3}}$ gives us that
    \begin{equation}
        \|\nabla^2 f\|^{2}_{L^{3}}\leq 5\|\nabla f\|_{L^{6}}\|\nabla^3 f\|_{L^{2}}+\|\nabla ^2 f\|_{L^{3}}^{-1}\int_{\partial\Omega}\left|\nabla^2f\right|^2 |\nabla f|dA_g.
    \end{equation}
    Now, if $\|\nabla ^2 f\|_{L^{3}}\ge 1$ then we find that
\begin{equation}
        \|\nabla^2 f\|^{2}_{L^{3}}\leq 5\|\nabla f\|_{L^{6}}\|\nabla^3 f\|_{L^{2}}+\int_{\partial\Omega}\left|\nabla^2f\right|^2 |\nabla f|dA_g,
    \end{equation}
    and if not, then $\|\nabla ^2 f\|_{L^{3}}<1$, which is a reasonable bound in its own right.
\end{proof}

By using the previous proposition we are now able to overcome the issue of a quadratic Hessian term in order to obtain an $L^2$ estimate on the third derivative of $u$ in terms of quantities we expect to be able to estimate.

\begin{thm}\label{Third-ThirdOrderEstimate}
Let $(M,g)$ be an oriented complete $n$-dimensional Riemannian manifold, $u$ a harmonic function, and $\Omega \subset M$ a bounded domain with smooth boundary $\partial \Omega$ and outward pointing unit normal vector $\nu$.
Then, for any $p >1$,
\begin{align}
\begin{split}
\int_{\Omega} |\nabla^3 u|^2dV_g &\le C\left(\int_{\Omega}|Rm|^{2p}+|Rc|^{2p}dV_g\right)^{1/p}\left(\int_{\Omega}|\nabla u|^{\frac{2p}{p-1}}dV_g\right)^{\frac{p-1}{p}}
\\&+C\max\left(2\left(\int_{ \Omega} |Rm|^3+|Rc|^3 dV_g\right)^{\frac{1}{3}},\left(\int_{ \Omega} |Rm|^3+|Rc|^3 dV_g\right)^{\frac{2}{3}}\|\nabla u\|_{L^6(\Omega)}^2\right)
\\&+C\left(\int_{ \Omega} |Rm|^3+|Rc|^3 dV_g\right)^{\frac{1}{3}}\max\left(1,\int_{\partial\Omega}|\nabla^2u|^{2}|\nabla u| dA_g\right)
\\&+\int_{\partial \Omega} \nabla\nabla\nabla u(\nabla\nabla u,\nu)-\nabla \nabla u(\Delta \nabla u, \nu)dA_g
\end{split}
\end{align}
\end{thm}
\begin{proof}
If we apply Cauchy-Schwarz to the result from Theorem \ref{FirstThirdOrderEstimate}, we find
\begin{align}\label{eq:3rd_diff_estimate}
\begin{split}
\int_{\Omega} |\nabla^3 u|^2dV_g&\le\int_{\Omega} 2|Rc||\nabla u| |\nabla^3 u| +|Rc|^2|\nabla u|^2 +|Rc||\nabla^2 u|^2dV_g
\\&+\int_{\Omega}|Rm||\nabla^3 u|| \nabla u|+|Rm||\nabla^2 u|^2dV_g
\\&+\int_{\partial \Omega} \nabla\nabla\nabla u(\nabla\nabla u,\nu)-\nabla \nabla u(\Delta \nabla u, \nu)dA_g.
\end{split}
\end{align}
Applying H\"{o}lder's inequality with $p > 1$, we find
\begin{align}
   \int_{\Omega} |Rc|^2|\nabla u|^2dV_g \le \left(\int_{\Omega}|Rc|^{2p} dV_g\right)^{1/p} \left(\int_{\Omega}|\nabla u|^{\frac{2p}{p-1}} dV_g\right)^{\frac{p-1}{p}} .
\end{align}

We can deal with the remaining terms which are not quadratic in the Hessian of $u$ by applying Young's inequality with $\varepsilon$ and  H\"{o}lder's inequality with $p > 1$ to find
\begin{align}
\int_{\Omega}&2|Rc||\nabla^3 u||\nabla u| +|Rm||\nabla^3 u|| \nabla u|  dV_g
\\&\le
\int_{\Omega}2\left(C(\varepsilon)|Rc|^2|\nabla u|^2+\varepsilon|\nabla^3 u|^2\right) +(C(\varepsilon)|Rm|^2| \nabla u|  ^2+\varepsilon|\nabla^3 u|^2dV_g
\\&=
\int_{\Omega}C'(\varepsilon)(|Rc|^2|\nabla u|^2+|Rm|^2| \nabla u|  ^2)+5\varepsilon|\nabla^3 u|^2dV_g
\\&\le C'(\varepsilon) \left(\int_{\Omega}|Rc|^{2p}+|Rm|^{2p}dV_g\right)^{1/p}\left(\int_{\Omega}|\nabla u|^{\frac{2p}{p-1}}dV_g\right)^{\frac{p-1}{p}}+\int_{\Omega}5\varepsilon|\nabla^3 u|^2dV_g.
\end{align}

For the terms which are quadratic in the Hessian we first apply H\'{o}lder's inequality with $p =\frac{3}{2}$ to find
\begin{align}
\int_{\Omega}(|Rm|+|Rc|)|\nabla^2 u|^2dV_g
&\le \left(\int_{ \Omega} |Rm|^{3}+|Rc|^{3} dV_g\right)^{\frac{1}{3}} \left(\int_{ \Omega} |\nabla^2 u|^3dV_g\right)^{\frac{2}{3}}
\\&\le \left(\int_{ \Omega} |Rm|^{3}+|Rc|^{3} dV_g\right)^{\frac{1}{3}} \max\left(1,\|\nabla^2u\|_{L^3}^2\right),
\end{align}
and now the problem is that one cannot expect to have control on $\|\nabla^2u\|_{L^3}^2$ and hence we will control this term using $\|\nabla u\|_{L^6}$ and $\|\nabla^{3}u\|_{L^{2}}$ as follows. 

From Proposition \ref{prop:HarmonicInterpolationIneq}, we have
\begin{equation}
    \begin{split}
        \max\left(1,\|\nabla^2u\|_{L^3(\Omega)}^2\right)&\leq \max\left(1,5\|\nabla u\|_{L^6(\Omega)}\|\nabla^3 u\|_{L^2(\Omega)}+\int_{\partial\Omega}\left|\nabla^{2}u\right|^{2}|\nabla u|\right)dA_g,
    \end{split}
\end{equation}
and hence
\begin{align}
&\int_{\Omega}(|Rm|+|Rc|)|\nabla\nabla u|^2dV_g
\\&\le \left(\int_{ \Omega} |Rm|^{3}+|Rc|^{3} dV_g\right)^{\frac{1}{3}} \max\left(1,5\|\nabla u\|_{L^6}\|\nabla^3 u\|_{L^2}+\int_{\partial\Omega}\left|\nabla^{2}u\right|^{2}|\nabla u|dA_g\right)
\\&\le \left(\int_{ \Omega} |Rm|^{3}+|Rc|^{3} dV_g\right)^{\frac{1}{3}} \left(\max\left(1,5\|\nabla u\|_{L^6}\|\nabla^3 u\|_{L^2}\right)+\max\left(1,\int_{\partial\Omega}\left|\nabla^{2}u\right|^{2}|\nabla u|dA_g\right)\right).
\end{align}
Now, we may apply Young's inequality to the first term in the last expression to find
\begin{align}
 &\left(\int_{ \Omega} |Rm|^{3}+|Rc|^{3} dV_g\right)^{\frac{1}{3}}\max\left(1,5\|\nabla u\|_{L^6}\|\nabla^3 u\|_{L^2}\right) 
 \\&\le \max\left(\left(\int_{ \Omega} |Rm|^{3}+|Rc|^{3} dV_g\right)^{\frac{1}{3}},C\|\nabla u\|_{L^6}^2\left(\int_{ \Omega} |Rm|^{3}+|Rc|^{3} dV_g\right)^{\frac{2}{3}}+\frac{1}{2}\|\nabla^3 u\|_{L^2}^2\right),
\\&\le \max\left(\left(\int_{ \Omega} |Rm|^{3}+|Rc|^{3} dV_g\right)^{\frac{1}{3}},C\|\nabla u\|_{L^6}^2\left(\int_{ \Omega} |Rm|^{3}+|Rc|^{3} dV_g\right)^{\frac{2}{3}}\right)
\\&+\max\left(\left(\int_{ \Omega} |Rm|^{3}+|Rc|^{3} dV_g\right)^{\frac{1}{3}},\frac{1}{2}\|\nabla^3 u\|_{L^2}^2\right) .
\end{align}
Notice that if $\|\nabla^3 u\|_{L^2}^2 \le 2\left(\int_{ \Omega} |Rm|^{3}+|Rc|^{3} dV_g\right)^{\frac{1}{3}}$ then we are done and if not then we also have our desired bound by subtracting off $\tfrac{1}{2}\|\nabla^3u\|_{L^{2}(\Omega)}$ from both sides of Equation \eqref{eq:3rd_diff_estimate} and then keeping track of the sequence of inequalities which follow from the above equation.
\end{proof}

%%%%%%%%%%%%%%%%%%%%%%%%%%%%%%%%%%%%%%%%%%%%%%%%%%%%%%
\section{Mass Estimates For Harmonic Functions}\label{sec-MassEstimates}

%%%%%%%%%%%%%%%%%%%%%%%%%%%%%%%%%%%%%%%%%%%%%%%%%%%%%%

In this section our goal is to use the estimates of the previous section combined with the mass formula to obtain H\"{o}lder control on Harmonic functions and their inner products. One should notice that each of the estimates of the previous section involve boundary terms which involve control in the asymptotically flat region and so we require control on these terms. Since the boundaries of the $M_{r}$ are in the asymptotically flat region and well understood, in this section we will apply the results of Section \ref{sec:int-est-harmonic-func} to asymptotically linear harmonic functions over the regions $M_r$. 

{\bf{Convention:}}
In the following statement and throughout this section, $r_0$ denotes the radius provided by the conclusions of Propositions \ref{p:asymptoticannulusdecay} and \ref{C3AsymptoticEstimate}.
Also, throughout this section, we will fix parameters $b,\Lambda, \kappa,\bar{m}>0$ and $\tau>\tfrac12$. 

\begin{thm}\label{AsymFlatL^2Estimate}
Let $(M,g)$ be a member of the family $\mathcal{M}(b,\tau,\bar{m},\Lambda,\kappa)$ and
let $u$ be an asymptotically linear harmonic function. Then, for any $r > r_0$, we have
\begin{align}
\left(\int_{M_r}|\nabla u|^6 dV_g\right)^{1/6}  
&\le C(r,r_0,b,\tau,\kappa,\Lambda,\bar{m}) ,
\\\left(\int_{M_r}|\nabla^2 u|^2 dV_g\right)^{1/2}  
&\le C(r,r_0,b,\tau,\kappa,\Lambda,\bar{m}).
\end{align}

\end{thm}
\begin{proof}
First, assume that $u$ is an asymptotically linear harmonic function from Proposition \ref{asymptotic-gradient-convergence}.
By combining Proposition \ref{asymptotic-gradient-convergence} with Proposition \ref{FirstOrderEstimate} we find that 
\begin{align}
    \|\nabla u\|_{L^2(M_r)}\le C_1(r,r_0,b,\tau,\kappa,\Lambda).\label{FirstGradL2Bound}
\end{align}
Now, by the mass formula, see Equation \eqref{ADM-formula}, we know 
\begin{align}\label{e:sobolevwithmass}
     4\int_{M_r}|\nabla \sqrt{|\nabla u|}|^2 dV_g\le \int_{M_r}\frac{|\nabla^2 u|^2}{|\nabla u|} dV_g\le 16 \pi m \le 16 \pi \bar{m}.
\end{align}
It follows from the fact that $M_r$ can be contained in $B_{2d}(x)$ for some $x\in M_1$ and $d=d(r,r_0,b,\tau,\bar m)$, the Sobolev inequality on metric balls, Theorem \ref{DirichletToSobolev},  and Equation \eqref{FirstGradL2Bound} that we have
\begin{align}
    \left(\int_{M_r}|\nabla u|^3 dV_g \right)^{1/6} \le C_2(r,b,\tau,\kappa,\Lambda,\bar{m}).
\end{align}

Similarly, by combining Theorem \ref{C3AsymptoticEstimate} and Theorem \ref{SecondOrderEstimate} with $p=3$, we find that
\begin{align}
    \|\nabla^2 u\|_{L^2(M_r)}\le C_3(r,b,\tau,\kappa,\Lambda,\bar{m}).
\end{align}
Hence by applying the fact that $|\nabla |\nabla u||^2 \le |\nabla^2 u|^2$ along with Theorem \ref{DirichletToSobolev} we find that
\begin{align}
    \|\nabla u\|_{L^6(M_r)}\le C_4(r,b,\tau,\kappa,\Lambda,\bar{m}).
\end{align}
Now, since any asymptotically linear harmonic function differs from the harmonic functions of Proposition \ref{asymptotic-gradient-convergence} by subtracting a constant, we obtain the desired result.
\end{proof}

Now we move to a bound on the Hessian of asymptotically linear harmonic functions in terms of the mass.

\begin{thm}\label{SecondOrderToZeroEstimate}
Let $(M,g)$ be in $\mathcal{M}(b,\tau,\bar{m},\Lambda,\kappa)$ and let $u$ be an asymptotically linear harmonic function.
Then, for $r>r_0$, $2 \le q \le 6$, and $\theta=\frac{6-q}{2q}$, the following hold
\begin{align}
\left(\int_{M_r}|\nabla^3 u|^2 dV_g\right)^{1/2} &\le C(r,r_0,b,\tau,\kappa,\Lambda,\bar{m}),
\\ \left(\int_{M_r}|\nabla^2 u|^q dV_g\right)^{1/q}  
&\le C(r,r_0,b,\tau,\kappa,\Lambda,\bar{m}) m^{\frac{\theta}{2}}.\label{qSecondOrderEstimate}
\end{align}
\end{thm}
\begin{proof}
% By Lemma \ref{NeumanToDirichletIsoperimetricEst} we have that
% \begin{align}
%     ID_{\frac{3}{2}}(B_r(x)) \ge \Lambda', \forall x \in M, r > 0,
% \end{align}
% and using Proposition \ref{prop:upper-lower-volume-bounds} with $R=2\Diam\left(M_{r_0}\right)$, we have that
% \begin{equation}
%     |B(x,r)|\leq K(b,\tau,\Lambda,r_{0})r^3
% \end{equation}
% for $x\in M_{r_0}$ and $r\leq\Diam M_{r_0}$.

If we combine Theorem \ref{C3AsymptoticEstimate}, Theorem \ref{AsymFlatL^2Estimate}, Theorem \ref{NeumantoSobolev}, Theorem \ref{Third-ThirdOrderEstimate}  with $p=\frac{3}{2}$, Remark \ref{rmrk:int_Rm_from_int_Ric}, and the fact that $M_r$ can be contained in $B_{2d}(x)$ for some $x\in M_1$ and $d=d(r,r_0,b,\tau,\bar m)$, then we find that
\begin{align}
\|\nabla^3 u\|_{L^2(M_r)} \le C_1(r,r_0,b,\tau,\kappa,\Lambda).
\end{align}

Now the fact that $|\nabla |\nabla^2 u||^2 \le |\nabla^3 u|^2$ together with Theorem \ref{NeumantoSobolev} implies that
\begin{align}
    \left(\int_{M_r}|\nabla^2 u|^6 dV_g\right)^{1/6} &\le C_2(r,r_0,b,\tau,\kappa,\Lambda)\left(\int_{M_r}|\nabla^3 u|^2 dV_g\right)^{1/2} \\
    &\le C_3(r,r_0,b,\tau,\kappa,\Lambda).
\end{align}
At this point we may use the mass formula with an $L^q$ interpolation inequality where $2 < q < 6$ and $\frac{1}{q}=\frac{\theta}{2}+\frac{1-\theta}{6}$ to see that
\begin{align}
\left(\int_{M_r}|\nabla^2 u|^q dV_g\right)^{1/q} &\le  \left(\int_{M_r}|\nabla^2 u|^2 dV_g\right)^{\frac{\theta}{2}}\left(\int_{M_r}|\nabla^2 u|^6 dV_g\right)^{\frac{1-\theta}{6}} \\
&\le C_4(r,r_0,b,\tau,\kappa,\Lambda)m^{\frac{\theta}{2}}.\label{qSecondOrderEstimate2}
\end{align}
\end{proof}

The above gives us an $L^{\infty}$ bound on $|\nabla u|$, which will be of great help later.
\begin{cor}\label{cor:C0-grad}
    Let $(M^3,g)$ be in $\mathcal{M}(b,\tau,\bar{m},\Lambda,\kappa)$ and let $u$ be an asymptotically linear harmonic function. Then, for any $r>r_{0}$, we have
    \begin{equation}
        \|\nabla u\|_{C^{0}(M_{r})}\leq C(r,r_{0},b,\tau,\kappa,\Lambda,\bar{m})
    \end{equation}
\end{cor}
\begin{proof}
    To begin, we may use Lemma \ref{AsymDiamEst} to find a $d=d(b,\tau,\Lambda,r)$ and $x_{0}$ in $M_{1}$ such that $B_{d}=B_{d}(x_{0})$ contains $M_{r}$, and we may pick $s=s(b,\tau,d+R)$ such that $B_{d+R}=B_{d+R}(x_{0})$ is contained in $M_{s}$. Then, from Theorem \ref{thm:Morrey's-inequality-balls} we have
    \begin{equation}
        \|\nabla u\|_{C_{R}^{0,\frac{p-n}{p}}(B_{d})}\leq C_1(b,\tau,\kappa,R,p)\left(\|\nabla u\|_{L^{1}(B_{d+R})}+\|\nabla^2u\|_{L^{p}(B_{d+R})}\right).
    \end{equation}
    Now, using H{\"o}lder's inequality and the results of Theorem \ref{SecondOrderToZeroEstimate} with $p=6$, we see that
    \begin{equation}
        \|\nabla u\|_{C_{R}^{0,\frac{1}{2}}(B_{d})}\leq C_2(b,\tau,\kappa,R,p)\left(|B_{d+R}|^{\frac{5}{6}}\|\nabla u\|_{L^{6}(B_{d+R})}+\|\nabla^2u\|_{L^{6}(B_{d+R})}\right)
    \end{equation}
From Theorem \ref{SecondOrderToZeroEstimate}, we see that the above is bounded by a constant $C_3(r,r_0,b,\tau,\bar{m},\kappa,\Lambda)$. Of course, this bounds the supremum of $|\nabla u|$ independently of $R$, and the result follows.
\end{proof}

Having established control on the $L^p$ norm of $\nabla^2u^i$ and on the $C^{0}$ norm of $|\nabla u^{i}|$, we may begin to analyze the inner products of the gradients of $u^i$. 

\begin{thm}\label{GradInnerProductToZeroEstimate}
Let $(M^3,g)$ be in $\mathcal{M}(b,\tau,\bar{m},\Lambda,\kappa)$ and let $u^{i},u^{j}$ be a pair of its asymptotically linear harmonic functions.
Then, for $r>r_0$, $1\le q\le6$, and $\theta=\tfrac{6-q}{2q}$, we have
 \begin{align}
 &\left(\int_{M_r} |\nabla g(\nabla u^i,\nabla u^j)|^q dV_g\right)^{\frac{1}{q}}\le C(r,r_0,b,\tau,\kappa,\Lambda,\bar{m},q) m^{\frac{\theta}{2}}.\label{GradInnerProductEstimate}
\end{align}
\end{thm}
\begin{proof}
We calculate that
\begin{equation}
    \int_{M_r}|\nabla g(\nabla u^{i},\nabla u^{j})|^qdV_g\leq C_1(q)\int_{M_r} |\nabla ^2u^i|^q|\nabla u^j|^q+|\nabla ^2u^j|^q|\nabla u^i|^qdV_g.
\end{equation}
Since both $|\nabla u^i|$ and $|\nabla u^{j}|$ are bounded above by some constant $C_2(r,r_{0},b,\tau,\kappa,\Lambda,\bar{m})$ by Corollary \ref{cor:C0-grad}, it follows that we have
\begin{equation}
    \int_{M_r} |\nabla g(\nabla u^{i},\nabla u^{j})|^qdV_g\leq C_1(q)C_2(r,r_{0},b,\tau,\kappa,\Lambda,\bar{m})^q\int_{M_r} |\nabla^2u^i|^q+|\nabla^{2}u^j|^qdV_g.
\end{equation}
By Theorem \ref{SecondOrderToZeroEstimate} and the above, we see that
\begin{equation}
    \|\nabla g(\nabla u^{i},\nabla u^{j})\|_{L^{q}}\leq 2^{\frac{1}{q}}C_1(q)C_2(r,r_{0},b,\tau,\kappa,\Lambda,\bar{m})m^{\frac{6-q}{4q}}.
\end{equation}

\end{proof}

\noindent Next we apply Theorem \ref{GradInnerProductToZeroEstimate} with our Morrey's inequality to show the inner products of the gradients converge to their average.

\begin{cor}\label{InnerProductToZeroEstimate}
Let $(M^3,g)$ be in $\mathcal{M}(b,\tau,\bar{m},\Lambda,\kappa)$ and let $u^{i},u^{j}$ be any two asymptotically linear harmonic functions. 
If $x\in M$, $r>r_0$, and $R, d>0$ are such that $B_{d+R}(x) \subset M_r$, then
 \begin{align}
 & \|g(\nabla u^i,\nabla u^j)-g(\nabla u^i,\nabla u^j)_{B_{d+R}(x)}\|_{C_{R}^{0,\gamma}(B_d(x))} \le C(r,r_0,b,\tau,\kappa,\Lambda,\bar{m})m^{\frac{1}{8}} .\label{InnerProductEstimate}
\end{align}
\end{cor}
\begin{proof}
% By Lemma \ref{NeumanToDirichletIsoperimetricEst} we have that
% \begin{align}
%     ID_{\frac{3}{2}}(B_d(x)) \ge C(\Lambda)>0,\quad \forall x \in M, d > 0,
% \end{align}
% and by Lemma \ref{DirichletIsoToLowerVolumeGrowth} we find
% \begin{align}
%     |B_d(x)| \ge C(\Lambda) d^3.
% \end{align}
% Additionally, from Proposition \ref{prop:upper-lower-volume-bounds} with $R=2\Diam M_{r_0}$, we have that
% \begin{equation}
%     |B_{d}(x)|\leq K(b,\tau,\Lambda,r_{0})d^3
% \end{equation}
If we choose $q=4$ and apply Theorem \ref{thm:Morrey's-inequality-balls} to Theorem \ref{GradInnerProductToZeroEstimate}, where $\theta = \frac{1}{4}$, we find
\begin{align}
    &\| g(\nabla u^i,\nabla u^j)-g(\nabla u^i,\nabla u^j)_{B_{d+R}(x)}\|_{C_{R}^{0,\gamma}(B_d(x))}
    \\&\le C_1(r,b,\tau,\kappa,\Lambda)\int_{B_{d+R}(x)}|g(\nabla u^i,\nabla u^j)-g(\nabla u^i,\nabla u^j)_{B_{d+R}(x)}|dV_g
    \\&+C_1(r,b,\tau,\kappa,\Lambda) \left(\int_{B_{d+R}(x)} |\nabla  g(\nabla u^i,\nabla u^j)|^4dV_g\right)^{1/4}
    \\&\le C_2(r,b,\tau,\kappa,\Lambda)\left(|B_{d+R}(x)|\left(\dashint_{B_{d+R}(x)} |\nabla  g(\nabla u^i,\nabla u^j)|^4dV_g\right)^{1/4}+m^{\frac{1}{8}}\right)
    \\&\le C_3(r,r_0, b, \tau,\kappa, \Lambda,\bar{m})m^{\frac{1}{8}},
\end{align}
where we use that $B_{d+R}(x) \subset M_r$ and Lemma \ref{DirichletToVolumeBound} in the last line.
\end{proof}

Now the goal is to put the estimates of the previous sections together in order to prove Theorem \ref{Thm-Mass Stability}. The first main step will be to show that since the average integral of $g(\nabla u^i,\nabla u^j)$ over a large coordinate annulus is close to $\delta^{ij}$ we can show that the average over $M_r$ is close as well. Then we notice that the previous estimate implies that $(u^1,u^2,u^3)$ are actually a global coordinate system on $M$, which will finish the proof.

\begin{proof}[Proof of Theorem \ref{Thm-Mass Stability}] 
Let $r_0$ be given by Proposition \ref{p:asymptoticannulusdecay} and suppose $r>r_0$ is a radius, which will be chosen successively larger in the arguments below. Denote $\mathcal{A}_{r_0,r}=M_r\setminus M_{r_0}$, and then consider
\begin{align}
    &\dashint_{M_r} |g(\nabla u^i,\nabla u^j)-\delta^{ij}| dV_g 
    \\&= \frac{1}{|M_r|}\left( \int_{\mathcal{A}_{r_0,r}} |g(\nabla u^i,\nabla u^j)-\delta^{ij}| dV_g+\int_{M_{r_0}} |g(\nabla u^i,\nabla u^j)-\delta^{ij}| dV_g  \right)
    \\&\le \frac{1}{|\mathcal{A}_{r_0,r}|}\left( \int_{\mathcal{A}_{r_0,r}} |g(\nabla u^i,\nabla u^j)-\delta^{ij}| dV_g+ \max_{1\le i \le 3} \left(\int_{M_{r_0}}|\nabla u^i|^2 dV_g\right) +|M_{r_0}|\right)
    \\&\le  \dashint_{\mathcal{A}_{r_0,r}} |g(\nabla u^i,\nabla u^j)-\delta^{ij}| dV_g+ \frac{|M_{r_0}|}{|\mathcal{A}_{r_0,r}|}C(r_0,b,\tau,\kappa,\Lambda),
\end{align}
where we have used Theorem \ref{AsymFlatL^2Estimate} and Proposition \ref{p:asymptoticannulusdecay} in the last line.
We know that since $\mathcal{A}_{r_0,r}$ is contained in the asymptotically flat region of $M$ that $|\mathcal{A}_{r_0,r}|\ge C_1(b,\tau)(r^3-r_0^3)$ and by Lemma \ref{DirichletToVolumeBound} we know $|M_{r_0}|\le C_2(r_0,b,\tau,\Lambda)r_0^3$ hence
\begin{align}
\begin{split}
    \dashint_{M_r} |g(\nabla u^i,\nabla u^j)-\delta^{ij}| dV_g 
   &\le \frac{1}{C_1(b,\tau)(r^3-r_0^3)} \int_{\mathcal{A}_{r_0,r}} |g(\nabla u^i,\nabla u^j)-\delta^{ij}| dV_g\\
   {}&\qquad+C_2(r_0,b,\tau,\Lambda) \frac{r_0^3}{(r^3-r_0^3)}.
\end{split}
   \label{AlmostAverageToZeroGlobally}
\end{align}
Notice that, by Proposition \ref{p:asymptoticannulusdecay}, choose $m$ small enough so that the first term on the right side of \eqref{AlmostAverageToZeroGlobally} is smaller than $\frac{\varepsilon}{2}$.
It follows that we may choose $r$ large enough so that the second term on the right side of \eqref{AlmostAverageToZeroGlobally} is less than $\frac{\varepsilon}{2}$.  So, given an $\varepsilon>0$, there is an $r_1=r_1(r_0,b,\tau,\Lambda,\varepsilon)$ and $m_1=m_1(r_0,b,\tau,\varepsilon)$ such that when $r>r_1$ and $m\leq m_1$, we have
\begin{align}
    \dashint_{M_r} |g(\nabla u^i,\nabla u^j)-\delta^{ij}| dV_g 
   &\le \varepsilon. \label{AverageToZeroGlobally}
\end{align}

Fix a point $x\in M_1$. By the first statement in Lemma \ref{AsymDiamEst} there is a radius $d=d(b,\tau,\Lambda,r)$ so that $M_r \subset B_{d}(x)$. We estimate
\begin{align}
   & \left|\dashint_{B_{d+1}(x)}g(\nabla u^i,\nabla u^j)dV_g-\delta^{ij}\right| 
    \\&=\dashint_{M_r} |g(\nabla u^i,\nabla u^j)-\delta^{ij}| dV_g
    \\&\quad +\left(\dashint_{B_{d+1}(x)}|g(\nabla u^i,\nabla u^j)-\delta^{ij}| dV_g-\dashint_{M_r}|g(\nabla u^i,\nabla u^j)-\delta^{ij}|dV_g\right)
    \\&\le\dashint_{M_r} |g(\nabla u^i,\nabla u^j)-\delta^{ij}| dV_g
    \\&\quad+\frac{1}{|B_{d+1}(x)|}\int_{B_{d+1}(x)\setminus M_r}|g(\nabla u^i,\nabla u^j)-\delta^{ij}| dV_g
     \\&\le\dashint_{M_r} |g(\nabla u^i,\nabla u^j)-\delta^{ij}| dV_g+\frac{C_3(r_0,b,\tau,\bar{m})|B_{d+1}(x)\setminus M_r|}{r^{\tau}|B_{d+1}(x)|}
     \\&\le\varepsilon+C_3(r_0,b,\tau,\bar{m})r^{-\tau}\label{e:nod-dependence}
\end{align}
where in the penultimate inequality we have made use of the pointwise estimate \eqref{e:asymptotic-gradient}. Notice that the estimate \eqref{e:nod-dependence} is independent of $d$. It follows there is an $r_2=r_2(r_0,b,\tau,\Lambda,\varepsilon)$ so that if $r>r_2$, we have
\begin{equation}\label{AverageInnerProductEstimate}
    \left|\dashint_{B_{d+1}(x)}g(\nabla u^i,\nabla u^j)dV_g-\delta^{ij}\right|\le2\varepsilon.
\end{equation}
% In a similar fashion we can find
% \begin{align}
% \delta^{ij} \le \dashint_{B_d(x)}|g(\nabla u^i,\nabla u^j)| dV +\dashint_{M_r} |g(\nabla u^i,\nabla u^j)-\delta^{ij}| dV+Cr^{-\tau}.
% \end{align}
% So that we see for $r$ chosen large enough and $m$ chosen small enough
% \begin{align}
%   \delta^{ij}-\varepsilon \le  \dashint_{B_d(x)}|g(\nabla u^i,\nabla u^j)| dV&\le \delta^{ij}+\varepsilon. \label{AverageInnerProductEstimate}
% \end{align}

Now, by the second statement of Lemma \ref{AsymDiamEst}, there is an $\hat{r}=\hat{r}(r, r_0,b,\tau,\Lambda)$ so that $B_{d+1}(x)\subset M_{\hat{r}}$. Combining Corollary \ref{InnerProductToZeroEstimate} and equation \eqref{AverageInnerProductEstimate}, we find that there is an $m_2=m_2(r,r_0,b,\tau,\kappa,\Lambda,\bar m)$ such that if $m<m_2$, then
\begin{align}
    \| g(\nabla u^i,\nabla u^j)-\delta^{ij}\|_{C_{1}^{0,\gamma}(B_d(x))}
    &\le  \| g(\nabla u^i,\nabla u^j)- g(\nabla u^i,\nabla u^j)_{B_{d+1}(x)}\|_{C_{1}^{0,\gamma}(B_d(x))}
    \\&\quad + |  g(\nabla u^i,\nabla u^j)_{B_{d+1}(x)}-\delta^{ij}|
    \\
    {}&\leq C(r,r_0,b,\tau,\kappa,\Lambda,\bar{m})m^{\tfrac18}+2\varepsilon\\
    &\le 3\varepsilon.\label{LocalHolderEst}
\end{align}
Notice that, using Proposition \ref{prop:equiv-Holder-norms} (taking $R=1$) the estimate in \eqref{LocalHolderEst}, we have the global H\"{o}lder estimate 
\begin{equation}\label{e:holderest1}
    \| g(\nabla u^i,\nabla u^j)-\delta^{ij}\|_{C^{0,\gamma}(B_d(x))}<6\varepsilon.
\end{equation}
On the other hand, Lemma \ref{asymptotic-gradient-convergence} implies there is an $r_3=r_3(r_0,b,\tau,\bar m)$ so that when $r\geq r_3$, we have
\begin{align}\label{e:holderest2}
    \| g(\nabla u^i,\nabla u^j)-\delta^{ij}\|_{C^{0,\gamma}(M \setminus M_r)}\le  \varepsilon.
\end{align}
Finally, we fix a value of $r\geq\max(r_2,r_3)$. This fixes the size of $m_2$ required for \eqref{e:holderest1} to hold. It follows that when $m$ is sufficiently small relative to $r_0,b,\tau,\kappa,$ and $\Lambda$, equations \eqref{e:holderest2} and \eqref{e:holderest2} combine to show the global H{\"o}lder closeness
\begin{equation}\label{e:globalholderest}
    \| g(\nabla u^i,\nabla u^j)-\delta^{ij}\|_{C^{0,\gamma}(M)}\le  7\varepsilon.    
\end{equation}

As a consequence of \eqref{e:globalholderest}, for sufficiently small $\varepsilon$, the map $U:M \rightarrow \R^3$ given by $U(x)=(u^1(x),u^2(x),u^3(x))$ becomes a diffeomorphism, giving a global coordinate system on $M \cong\R^3$. Then, in these coordinates $g^{ij}=g(\nabla u^i,\nabla u^j)$ and \eqref{e:globalholderest} becomes $\| g^{ij}-\delta^{ij}\|_{C^{0,\gamma}(M)}\le  7\varepsilon.$ By perhaps taking $\varepsilon$ smaller, we may assume  $\varepsilon<\frac{1}{100}$. Now since $g^{-1}$ is within $7\varepsilon$ of the identity matrix in this global coordinate system, its inverse, that is $\{g_{ij}\}$, is within $14\varepsilon$ of the identity in these coordinates, yielding the desired estimate. 
\end{proof}

%%%%%%%%%%%%%%%%%%%%%%%%%%%%%%%%%%%%%%%%%%%%%%%%%%%%%%

\section{Estimates of Harmonic Functions on the Torus} \label{sec-Torus}
In a similar way that we used the mass formula to control the geometry of asymptotically flat manifolds in terms of their mass, we can use the formula Theorem \ref{t:Stern} to control the geometry of $3$-manifolds in $\mathcal{N}(\Lambda,V,\kappa)$ in terms of the $L^{1}$ norm of the negative part of their scalar curvature. 
Before approaching Theorem \ref{Thm-Torus Stability}, some preliminary observations are required. The following is an important result which comes from combining information about isoperimetric profiles with the result \cite[Theorem 5.4]{Petersen-97}.
\begin{prop}\label{prop:Petersen-harmonic-norm-bound}
    Let $\Lambda,V,\kappa>0$. 
    Then $\mathcal{N}\left(\Lambda,V,\kappa\right)$ is pre-compact in the $C^{0,\gamma}$-topology for any $\gamma\in(0,1)$.
    %There exists an $r=r(Q,\kappa,\Lambda)$ so that the following holds: If $(M,g)\in \mathcal{N}\left(\Lambda,V,\kappa\right)$, then
    %\begin{equation}\label{e:petersenw32}
    %    \|(M,g)\|_{3,2,r}\leq Q,
    %\end{equation}
    %where $\|*\|_{3,2,r}$ is Petersen's harmonic norm (see %\cites{Petersen-2006,Petersen-97}). 
    In particular, the underlying manifolds in  $\mathcal{N}\left(\Lambda,V,\kappa\right)$ range over a finite number of diffeomorphism types. 
\end{prop}
\begin{proof}
    According to Lemma \ref{NeumanToDirichletIsoperimetricEst}, we see that for any region $\Omega\subset M$ such that $|\Omega|\leq\frac{1}{2V}$, we have
    \begin{equation}\label{e:IDtoIN}
        ID_{\frac{3}{2}}\left(\Omega\right)\geq IN_{\frac{3}{2}}\left(M,g\right)\geq\Lambda.
    \end{equation}
    In particular, by Corollary \ref{volume-growth-above} we see there is an $R(\kappa,V)$ such that if $r\leq R(\kappa,V)$, then $|B_r(x)|\leq\frac{1}{2V}$. Combining this with \eqref{e:IDtoIN}, Lemma \ref{DirichletIsoToLowerVolumeGrowth}, and another application of Corollary \ref{volume-growth-above}, we obtain the following control above and below for $r\leq R(\kappa,V)$
    \begin{equation}
        \frac{r^{3}}{C(\Lambda,\kappa)}\leq|B_r(x)|\leq C(\Lambda,\kappa)r^3.
    \end{equation}
    
    This volume growth property shows that $\mathcal{N}\left(\Lambda,V,\kappa\right)$ satisfies the assumptions of \cite[Theorem 5.4]{Petersen-97}, allowing us to conclude that this class has bounded harmonic $\|*\|_{3,2,r}$-norm, see \cites{Petersen-2006,Petersen-97} for the definition of this norm. As a consequence, this class is precompact in the $C^{0,\gamma}$-topology for any $\gamma\in(0,1)$. Since a sequence of manifolds converging in the $C^{0,\gamma}$ sense must have a tail whose underlying manifolds are diffeomorphic, the class $\mathcal{N}\left(\Lambda,V,\kappa\right)$ supports only finitely many diffeomorphism types.
\end{proof}

For each diffeomorphism type $M$ supporting a metric in $\mathcal{N}\left(\Lambda,V,\kappa\right)$, fix once and for all three closed differential forms $\omega_1,\omega_2,\omega_3\in \Omega^1(M)$ so that $\int_M\omega_1\wedge\omega_2\wedge\omega_3=1$. The existence of such forms follows from condition $(5)$ in Definition \ref{def:toric}. Proposition \ref{prop:Petersen-harmonic-norm-bound} allows us to uniformly control the $L^2$-norm of the forms $\{\omega_i\}_{i=1}^3$.
\begin{cor}\label{cor:uniformly-controlled-cohomology}
    Let $\Lambda,V,\kappa>0$. Then there is a constant $C=C(\Lambda, V,\kappa)$ so that the following holds: if $(M,g)\in \mathcal{N}\left(\Lambda,V,\kappa\right)$, then
    \begin{equation}
        \|\omega_{i}\|_{L^{2}}\leq C,\qquad i=1,2,3.
    \end{equation}
\end{cor}
\begin{proof}
    We argue by contradiction, using the $C^{0,\gamma}$, for $\gamma\in(0,1)$, pre-compactness of $\mathcal{N}\left(\Lambda,V,\kappa\right)$ mentioned in the proof of Proposition \ref{prop:Petersen-harmonic-norm-bound}. Assuming there is no constant $C$ uniformly bounding $\|\omega_{i}\|_{L^{2}}$ for all $(M,g)$ in $\mathcal{N}\left(\Lambda,V,\kappa\right)$, we may find a sequence $\{(M_l,g_{l})\}^{\infty}_{l=1}$ so that
    \begin{equation}
        \lim_{l\rightarrow\infty}\|\omega_{i}\|_{L^{2}(M_l,g_{l})}=\infty
    \end{equation}
    for $i=1,2,3$. As a consequence of $C^{0,\gamma}$ pre-compactness, we may pass to a subsequence and assume $M_l$ are diffeomorphic to a single manifold $M$ and that $\{(M,g_{l})\}$ converges in the $C^{0,\gamma}$ topology to a limiting Riemannian manifold $(M,g_\infty)$. 
    By the nature of this convergence, $g_{\infty}$ is metric of $C^{0,\gamma}$ regularity. Since the $1$-forms $\omega_i$ are independent of the metric and smooth, that there exists a constant $K$ such that $\max_M|\omega_{i}|_{g_\infty}\leq K$. It follows from the Dominated Convergence Theorem that
    \begin{equation}
        \lim_{l\rightarrow\infty}\|\omega_{i}\|_{L^{2}(g_{l})}=\|\omega_{i}\|_{L^{2}(g_{\infty})}\leq KV,
    \end{equation}
    where the last inequality follows from the fact that $\Vol_{g_{\infty}}(M)\leq V$. This is a contradiction, and we conclude the desired result.
\end{proof}

We are now prepared to prove the stability result Theorem \ref{Thm-Torus Stability}.

\begin{proof}[Proof of Theorem \ref{Thm-Torus Stability}]
%    We will prove the claimed $C^{0,\gamma}$, for some $\gamma\in(0,1)$, estimate for the rescaled metric $\bar{g}=\Vol_g(M)^{-\frac{2}{3}}g$. This is without loss of generality -- if $\|\bar{g}-g_{F}\|_{C^{0,\gamma}}<\varepsilon$ for some flat metric $g_F$, then $\|g-\Vol_g(M)^{\frac{2}{3}}g_F\|_{C^{0,\gamma}}<V^{\frac{2}{3}}\varepsilon$ where we note that $\Vol_g(M)^{\frac{2}{3}}g_F$ is again flat. 
%    Moreover, the curvature and Sobolev constant bounds of $\bar{g}$ are related to $\Lambda$ and $\kappa$ as follows: We have 
%    \begin{equation}
%        \int_M\left|Rc_{\bar{g}}\right|^{p}_{\bar{g}}dV_{\bar{g}}=\Vol_g(M)^{\frac{2}{3}p-1}\int_M\left|Rc(g)\right|_{g}^{p}dV_{g}
%    \end{equation}
%    and so $\|Rc_{\bar{g}}\|_{L^{p}}\leq V^{\frac{2}{3}-\frac{1}{p}}\kappa$.
%    Also, $IN_{\frac{3}{2}}\left(M,\bar{g}\right)=IN_{\frac{3}{2}}\left(M,g\right)$ and so $(M,\bar{g})\in\mathcal{N}(\Lambda,1,V^{\frac{2}{3}-\frac{1}{p}}\kappa)$. Lastly, $\int_M|R^-_{\bar{g}}|dV_{\bar{g}}=\Vol_g(M)^{-\frac13}\int_M|R^-_{g}|dV_g$
    
    Observe that the integral bound on $\mathrm{Rc}_g$ implies that $\|R^{-}_{g}\|_{L^3}$ is bounded uniformly by $\kappa$. Using H{\"o}lder's inequality, it follows that $\|R^{-}_{g}\|_{L^2}$ is bounded in terms of $V$ and $\kappa$, a fact we will make use of below. Let $\{u^i\}_{i=1}^3$ be harmonic maps to $S^1$ so that the pull-backs $du^i:= (u^i)^*d\theta$ are cohomologous to the fixed forms $\{\omega_i\}_{i=1}^3$. Since harmonic forms minimize the $L^2$ norm in their cohomology class, Corollary \ref{cor:uniformly-controlled-cohomology} implies that
    \begin{equation}
        \|du^{i}\|_{L^{2}}\leq C_1(\Lambda,V,\kappa).
    \end{equation}
    
    At this point, we are ready to apply the arguments from Sections \ref{sec-IntegralEstimates} and \ref{sec-MassEstimates}. Let us begin by applying the integral identity \eqref{Scalar-formula} to find
    \begin{equation}\label{e:sterninproof}
        \int_M R^{-}(g)|du^i|dV_{\bar g}\geq\int_M\frac{|\nabla du^i|^2}{|du^i|}dV_{g},\qquad i=1,2,3.
    \end{equation}
    Similar to \eqref{e:sobolevwithmass}, we may combine \eqref{e:sterninproof}, H{\"o}lder's inequality, and the fact that $R^{-}(g)$ is $L^2$ bounded so that we may apply the Sobolev inequality to $\sqrt{|du^i|}$, yielding $\|du^{i}\|_{L^{3}}\leq C_2(\Lambda,V,\kappa)$. Now Theorem \ref{SecondOrderEstimate} with $\Omega=M$ shows that $\|\nabla du^{i}\|_{L^{2}}\leq C_3(\Lambda,V,\kappa)$. In turn, the uniform Sobolev constant yields $\|du^{i}\|_{L^{6}}\leq C_4(\Lambda,V,\kappa)$.
    
    Now, we return to the integration by parts
    \begin{equation}
        \begin{split}
            \int_M|\nabla du^{i}|^{3}dV_{g}&=-\int_M\nabla du^{i}\left(du^{i},\nabla |\nabla du^{i}|\right)+|\nabla du^{i}|\nabla^{*}\nabla du^{i}(du^{i})dV_{g}
            \\
            &\leq2\int_M\left|du^{i}\right|\left|\nabla du^{i}\right|\left|\nabla^2du^{i}\right|dV_{g}.
        \end{split}
    \end{equation}
    We may now use the generalized H{\"o}lder's inequality with exponents $6$, $3$, and $2$, respectively, to find
    \begin{equation}
        \int_M|\nabla du^{i}|^{3}dV_{g}\leq 2\|du^{i}\|_{L^{6}}\|\nabla du^{i}\|_{L^{3}}\|\nabla^2du^{i}\|_{L^2}.
    \end{equation}
    Dividing out both sides of the above by $\|\nabla du^{i}\|_{L^{3}}$ gives us that
    \begin{equation}\label{e:2ndordertorus}
        \|\nabla du^{i}\|^{2}_{L^{3}}\leq2\|du^{i}\|_{L^{6}}\|\nabla^2du^{i}\|_{L^{2}}.
    \end{equation}
    As in the proof of Theorem \ref{Third-ThirdOrderEstimate}, \eqref{e:2ndordertorus} can be used to show that $\|\nabla^2du^i\|_{L^2}\leq C_5(\Lambda,V,\kappa)$. Since $\|\nabla du^{i}\|_{L^2}\leq C_3(\Lambda,V,\kappa)$, the above allows us to apply the Sobolev inequality to $|\nabla du^{i}|$, yielding $\|\nabla du^{i}\|_{L^6}\leq C_6(\Lambda,V,\kappa)$. Finally, this last bound may be used with Theorem \ref{SobolevToMorrey}, showing the pointwise control
    \begin{equation}
        \left|du^{i}\right|_{g}\leq C_7(\Lambda,V,\kappa).
    \end{equation}
    As a second consequence of the $L^6$ bound on $|\nabla du^i|$, there is the following interpolation inequality, for any $q\in[2,6]$,
    \begin{equation}\label{e:hessianLqtorus}
        \|\nabla du^{i}\|_{L^{q}}\leq\|\nabla du^{i}\|_{L^{2}}^{\theta}\|\nabla du^{i}\|_{L^{6}}^{1-\theta}\leq C_8(\theta,\Lambda,V,\kappa)\|R^{-}(g)\|^{\frac{\theta}{2}}_{L^{1}},
    \end{equation}
    where $\tfrac{1}{q}=\tfrac{\theta}{2}+\tfrac{1-\theta}{6}$.

    Since the wedge product of $\{\omega_i\}_{i=1}^3$ generates the top integral cohomology of $M$, 
    \begin{equation}
        \int_M du^{1}\wedge du^2\wedge du^{3}=1.
    \end{equation}
    %Together with our upper bound on $\left|d u^i\right|$, this gives the following lower bound on $\|d u^i\|_{L^{1}}$:
    %\begin{equation}
    %    \begin{split}
    %        \int\left|\nabla u^{i}\right|&\geq\frac{1}{C(\Lambda,\kappa,p)^2}\int\left|\nabla u^{1}\right|\left|\nabla u^{2}\right|\left|\nabla u^{3}\right|
    %        \\ 
    %        &\geq\frac{1}{C(\Lambda,\kappa,p)^2}\int du^{1}\wedge du^{2}\wedge du^{3}
    %        \\
    %        &=\frac{1}{C(\Lambda,\kappa,p)^2}.
    %    \end{split}
    %\end{equation}
    %where we have used Hadamard's inequality for matrices. 
    In addition, we have
    \begin{equation}
        \nabla\left(du^{1}\wedge du^2\wedge du^{3}\right)=\nabla du^{1}\wedge du^2\wedge du^{3}+du^1\wedge\nabla du^2\wedge du^3+du^{1}\wedge du^2\wedge\nabla du^3,
    \end{equation}
    so our $C^0$ control on $du^i$ allows us to apply Hadamard's inequality to find
    \begin{equation}\label{e:prevolcontrol}
        \begin{split}
            &\left|\nabla g\left(\nabla u^{i},\nabla u^{j}\right)\right|\leq C_3(\Lambda,V,\kappa)\left(\left|\nabla^2 u^{i}\right|+\left|\nabla^2u^{j}\right|\right)
            \\
          &\left|\nabla\left(du^{1}\wedge du^2\wedge du^{3}\right)\right|\leq C_4(\Lambda,V,\kappa)\left(\left|\nabla du^1\right|+\left|\nabla du^2\right|+\left|\nabla du^3\right|\right),
        \end{split}
    \end{equation}
    where we denote by $\nabla u^i$ the vector field corresponding to $du^i$.
    Considering the non-negative function $f$ on $M$ defined by
    \begin{equation}
        du^{1}\wedge du^2\wedge du^{3}=fdV_{g},
    \end{equation}
    then \eqref{e:hessianLqtorus} and \eqref{e:prevolcontrol} imply
    \begin{equation}\label{e:volformcontrol}
    \begin{split}
        &\|f\|_{L^{\infty}}\leq C_5(\Lambda,V,\kappa),
        \\
        &\|\nabla f\|_{L^{q}}\leq C_6(\Lambda,V,\kappa)\left(\|\nabla du^1\|_{L^{q}}+\|\nabla du^2\|_{L^{q}}+\|\nabla du^3\|_{L^{q}}\right)
    \end{split}
    \end{equation}
    for any $q\in[2,6]$.
    
    Combining \eqref{e:prevolcontrol} and \eqref{e:volformcontrol}, and applying Theorem \ref{SobolevToMorrey}, we find
    \begin{equation}\label{e:fcontrolc0}
    \begin{split}
        &\|f-\frac{1}{|M|}\|_{C^{0,\gamma}}\leq C_7(\Lambda,V,\kappa)\left(\|\nabla du^1\|_{L^{q}}+\|\nabla du^2\|_{L^{q}}+\|\nabla du^3\|_{L^{q}}\right)
        \\
        &\left\| g\left(\nabla u^{i},\nabla u^{j}\right)-g\left(\nabla u^{i},\nabla u^{j}\right)_{M}\right\|_{C^{0,\gamma}}\leq C_8(\Lambda,V,\kappa)\|\nabla du^{i}\|_{L^{q}},
    \end{split}
    \end{equation}
    where $\gamma=\frac{q-3}{q}$. A technical remark is required concerning the second line of \eqref{e:fcontrolc0}. Note that technically Theorem \ref{SobolevToMorrey} only estimates the $C^{0,\gamma}_R$ norm. However, in the present context, the $C^{0,\gamma}_R$ norm of a function differs from its $C^{0,\gamma}$ norm only by the uniformly controlled $R=R(\kappa, V)$ and the $C^0$ norm of that function. 
    
    As an immediate consequence of the first line of \eqref{e:fcontrolc0}, if  $||R^-_{g}||_{L^1}$ is sufficiently small relative to $\Lambda,\kappa,$ and $V$, $f$ cannot vanish since $\tfrac{1}{|M|}\geq\tfrac{1}{V}$. Therefore, the map $\Psi:M\to T^3$ given by 
    \begin{equation}
    \Psi(x):=(u^{1}(x),u^{2}(x),u^{3}(x))
    \end{equation}
    is a local diffeomorphism. Since the degree of $\Psi$ is $1$, it is in fact a global diffeomorphism.
    Now consider the coefficients
    \begin{equation}
        a^{ij}=g\left(\nabla u^{i},\nabla u^{j}\right)_{M}=\dashint_Mg\left(\nabla u^{i},\nabla u^{j}\right)dV_{{g}}.
    \end{equation}
    In light of \eqref{e:fcontrolc0}, the fields $\left\lbrace\nabla u^{i}\right\rbrace_{i=1}^{3}$ forms a frame for $TM$ if $||R^-_{g}||_{L^1}$ is sufficiently small relative to $\Lambda,\kappa,$ and $V$.
    Consequently, the constant symmetric tensor on the torus $\mathbb{T}^3$ given by 
    \begin{equation}
        \delta_{F}^{-1}=a^{ij}\frac{\partial}{\partial\theta^{i}}\otimes\frac{\partial}{\partial\theta^{j}}
    \end{equation}
   where $\theta^i$ denotes the angular coordinate of the $i^{\mathrm{th}}$ circle in $\mathbb{T}^3$, defines a flat metric $\delta_F$ for small $\|R^-(g)\|_{L^{1}}$.
   %Indeed, since $f$ is nowhere zero, it follows that $\left\lbrace\nabla u^{i}\right\rbrace_{i=1}^{3}$ forms a basis for any $x$ in $\mathbb{T}^3$. Then consider $|a^{ij}-\bar{g}(\nabla u^{i},\nabla u^{j})(x)|$ for any $x$ in $\mathbb{T}^3$. So, for small enough $\|R(\bar{g})\|_{L^{1}}$, Theorem \ref{SobolevToMorrey} and the fact that at each point $\bar{g}(\nabla u^{i},\nabla u^{j})$ is non-singular shows that $a^{ij}$ must also be non-singular. 
   Now, in the frame $\{\nabla u^i\}_{i=1}^3$ we may express $g^{ij}=g\left(\nabla u^{i},\nabla u^{j}\right)$.
    % \begin{equation}
    %     \bar{g}=\left<\nabla u^{i},\nabla u^{j}\right>_{\bar{g}}\frac{\partial}{\partial u^i}\frac{\partial}{\partial u^{j}}.
    % \end{equation}
    Finally, set $g_{F}=\Psi^*\delta_{F}.$
    Applying the second line of \eqref{e:fcontrolc0} with Theorem \ref{SobolevToMorrey}, we see that
    \begin{equation} \label{e:invmetricest}
        \|g^{ij}-g_{F}^{ij}\|_{C^{0,\gamma}(M)}<\frac{\varepsilon}{2}
    \end{equation}
    when $\|R^{-}(g)\|_{L^{1}}$ is small enough relative to $\Lambda,V,$ and $\kappa$. Similar to the final lines in the proof of Theorem \ref{Thm-Mass Stability}, inequality \eqref{e:invmetricest} implies $\|g_{ij}-\left(g_{F}\right)_{ij}\|_{C^{0,\gamma}}<\varepsilon$ for sufficiently small $\varepsilon$, finishing the argument.
\end{proof}

Lastly, we show Corollary \ref{Cor-SequentialTorusStability}.

\begin{proof}[Proof of Corollary \ref{Cor-SequentialTorusStability}]
In light of Theorem \ref{Thm-Torus Stability}, a diagonal argument shows that it suffices to show that the collection of flat Riemannian $3$-tori in $\mathcal{N}(\Lambda, V,\kappa)$ is compact in the $C^{0,\gamma}$-topology. This follows, for instance, by Proposition \ref{prop:Petersen-harmonic-norm-bound}. Indeed, one may fix coordinates on $\mathbb{T}^3$ and represent any flat metric $g$ as a constant matrix $g_{ij}d\theta^i\otimes d\theta^j$ .
\end{proof}

In light of Theorem \ref{Thm-Torus Stability}, a diagonal argument shows that it suffices to establish the following.

\begin{prop}
Given $V,\Lambda>0$ and $\alpha\in[1,\tfrac{3}{2}]$, the collection of flat Riemannian $3$-tori with Neumann $\alpha$-isoperimetric constant at least $\Lambda$ and volume within $[V^{-1},V]$ is compact in the $C^{\infty}$-topology (in particular in $C^{0,\gamma}$).
\end{prop}

\begin{proof}
Suppose $\mathcal{T}\in \mathcal{N}(\Lambda, V,\kappa)$ is a flat torus, which we may be represent as a quotient $\mathbb{R}^3/\Gamma$ of $3$-space by a lattice $\Gamma$ generated by three vectors $\{X_i\}_{i=1}^3$. The main goal is to show $(1)$ the lengths of $X_i$ are bounded above and away from $0$ in terms of $\Lambda,V$ and $(2)$ the angles between $X_i,X_j$ are bounded away from $0$ and $\pi$ in terms of $\Lambda,V$. Upon accomplishing this, one obtains the desired sequential compactness of flat tori in $\mathcal{N}(\Lambda, V,\kappa)$ by the following argument: given a family of such tori, one applies $(1)$ to find a subsequential limit of the underlying vectors to three non-zero vectors, then noting that $(2)$ ensures these vectors form a lattice of full rank. The maps to the limiting torus which yield the $C^\infty$ convergence are given by linear mappings, taking the lattice elements to their associated limits.

First, we'll need a few computational observations. The volume of $\mathcal{T}$ is computed by $|X_1\wedge X_2\wedge X_3|$. Meanwhile, the planes spanned by $X_i,X_j$ descend to $2$-dimensional tori with areas $|X_i\wedge X_j|$. Using one of these tori and an appropriate translate, one can cut $\mathcal{T}$ into two pieces of equal volumes and conclude
\begin{equation}\label{e:flattoriiso}
    |X_i\wedge X_j|\geq\frac{\Lambda}{2}|X_1\wedge X_2\wedge X_3|^{\frac1\alpha}
\end{equation}
for $i\neq j$.

For $r>0$, consider the cubical domains $\mathcal{C}_r\subset \mathcal{T}$ given by 
\begin{equation}
    \mathcal{C}_r=\left\{aX_1+b\frac{X_2}{|X_2|}+c\frac{X_3}{|X_3|}\colon a\in[0,1],\;b\in[0,r],\; c\in[0,r]\right\}.
\end{equation}
Notice that the faces of $\mathcal{C}_r$ where $a=0$ and $a=1$ are identified by the lattice $\Gamma$.
When $r<\min(|X_2|,|X_3|)$, one may compute
\begin{equation}
    |\mathcal{C}_r|=r^2\frac{|X_1\wedge X_2\wedge X_3|}{|X_2||X_3|},\quad |\partial\; \mathcal{C}_r|=2r\frac{|X_1\wedge X_2|}{|X_2|}+2r\frac{|X_1\wedge X_3|}{|X_3|}.
\end{equation}
On the other hand, $\partial_r|\mathcal{C}_r|=\tfrac12|\partial \mathcal{C}_r|$. Combining these three equations, we have 
\begin{equation}\label{e:linearalgident}
    \frac{|X_1\wedge X_2|}{|X_2|}+\frac{|X_1\wedge X_3|}{|X_3|}=2\frac{|X_1\wedge X_2\wedge X_3|}{|X_2||X_3|}.
\end{equation}
After performing straight-forward manipulations, the identity \eqref{e:linearalgident} used with two applications of \eqref{e:flattoriiso} yields
\begin{equation}
    |X_1\wedge X_2\wedge X_3|^{\tfrac{\alpha-1}{\alpha}}\geq\frac{\Lambda}{4}\left(|X_3|+|X_2|\right).
\end{equation}
Since this inequality is symmetric in $X_1,X_2,X_3$, we may conclude an upper bound $|X_i|\leq 4 V^{\tfrac{\alpha-1}{\alpha}}/\Lambda$ for $i=1,2,3$. 
Using this, a lower bound for the lengths may be obtained by 
\begin{equation}
    \frac1V\leq|X_1\wedge X_2\wedge X_3|\leq |X_1||X_2||X_3|\leq |X_i|\frac{16 V^{\tfrac{2\alpha-2}{\alpha}}}{\Lambda^2}
\end{equation}
for any $i=1,2,3$. To summarize these last two observations, there is a constant $C(\Lambda,V)>0$ so that $C(\Lambda,V)^{-1}\leq |X_i|\leq C(\Lambda,V)$. 

All that is left is to establish that the angles $\theta_{ij}$ between $X_i,X_j$ are bounded away from $0$ and $\pi$ in terms of $\Lambda, V$. This is accomplished by combining the isoperimetric inequality with the length bounds as follows
\begin{align}
    \sin^2(\theta_{ij})&=1-\cos^2(\theta_{ij})\\
    {}&=\frac{|X_i|^2|X_j|^2-\langle X_i,X_j\rangle^2}{|X_i|^2|X_j|^2}\\
    {}&=\frac{| X_i\wedge X_j|^2}{|X_i|^2|X_j|^2}\\
    {}&\geq C(\Lambda,V)^{-4}\left(\frac{\Lambda}{2 V^{\frac1\alpha}}\right)^2.
\end{align}

\end{proof}

\bibliography{bibliography}

% \bib, bibdiv, biblist are defined by the amsrefs package.
\begin{bibdiv}
\begin{biblist}

\bib{AllenBryden21}{misc}{
      author={Allen, Brian},
      author={Bryden, Edward},
       title={Sobolev inequalities and convergence for riemannian metrics and
  distance functions},
   publisher={arXiv},
        date={2021},
         url={https://arxiv.org/abs/2112.05105},
}

\bib{Anderson-Cheeger}{article}{
      author={Anderson, Michael~T.},
      author={Cheeger, Jeff},
       title={Diffeomorphism finiteness for manifolds with {R}icci curvature
  and {$L^{n/2}$}-norm of curvature bounded},
        date={1991},
        ISSN={1016-443X},
     journal={Geom. Funct. Anal.},
      volume={1},
      number={3},
       pages={231\ndash 252},
         url={https://doi.org/10.1007/BF01896203},
      review={\MR{1118730}},
}

\bib{AHMPPW1}{article}{
      author={Allen, B.},
      author={Hernandez, L.},
      author={Parise, D.},
      author={Payne, A.},
      author={Wang, S.},
       title={Warped tori with almost non-negative scalar curvature},
        date={2018},
     journal={Geometriae Dedicata},
      volume={200},
      number={2},
}

\bib{Allen18}{article}{
      author={Allen, Brian},
       title={Imcf and the stability of the pmt and rpi under $l^2$
  convergence},
        date={2017},
     journal={Annales Henri Poincar\'e},
      volume={19},
      number={1},
}

\bib{Allen-Conformal-Torus}{article}{
      author={Allen, B.},
       title={Almost non-negative scalar curvature on riemannian manifolds
  conformal to tori},
        date={2021},
     journal={Journal of Geometric Analysis},
      volume={31},
       pages={11190\ndash 11213},
}

\bib{Anderson-Ricci}{article}{
      author={Anderson, Michael~T.},
       title={Convergence and rigidity of manifolds under {R}icci curvature
  bounds},
        date={1990},
        ISSN={0020-9910},
     journal={Invent. Math.},
      volume={102},
      number={2},
       pages={429\ndash 445},
         url={https://doi.org/10.1007/BF01233434},
      review={\MR{1074481}},
}

\bib{Finster-Bray-99}{article}{
      author={Bray, Hubert},
      author={Finster, Felix},
       title={Curvature estimates and the positive mass theorem},
        date={2002},
     journal={Comm. Anal. Geom},
       pages={291\ndash 306},
}

\bib{BKKS19}{article}{
      author={Bray, Hubert~L.},
      author={Kazaras, Demetre~P.},
      author={Khuri, Marcus~A.},
      author={Stern, Daniel~L.},
       title={Harmonic functions and the mass of 3-dimensional asymptotically
  flat riemannian manifolds},
        date={2022},
     journal={Journal of geometric analysis},
      volume={32},
      number={184},
      eprint={https://arxiv.org/abs/1911.06754},
}

\bib{BKS}{article}{
      author={Bryden, Edward},
      author={Khuri, Marcus},
      author={Sormani, Christina},
       title={Stability of the spacetime positive mass theorem in spherical
  symmetry},
        date={2021},
     journal={Journal of geometric analysis},
      volume={31},
       pages={4191\ndash 4239},
      eprint={https://arxiv.org/pdf/1906.11352.pdf},
}

\bib{Bryden20}{article}{
      author={Bryden, Edward},
       title={Stability of the positive mass theorem for axisymmetric
  manifolds},
        date={2020},
     journal={Pacific journal of mathematics},
      volume={305},
      number={1},
       pages={89\ndash 152},
      eprint={https://arxiv.org/abs/1806.02447},
}

\bib{Basilio-Sormani}{article}{
      author={Basilio, J.},
      author={Sormani, C.},
       title={Sequences of three dimensional manifolds with positive scalar
  curvature},
        date={2019},
         url={https://arxiv.org/abs/1911.02152},
}

\bib{Cheeger-Colding-1}{article}{
      author={Cheeger, Jeff},
      author={Colding, Tobias~H.},
       title={On the structure of spaces with {R}icci curvature bounded below.
  {I}},
        date={1997},
        ISSN={0022-040X},
     journal={J. Differential Geom.},
      volume={46},
      number={3},
       pages={406\ndash 480},
         url={http://projecteuclid.org/euclid.jdg/1214459974},
      review={\MR{1484888}},
}

\bib{Chu-Man-Chun22}{article}{
      author={Chu, Jianchun},
      author={Lee, Man-Chun},
       title={K{\"a}hler tori with almost non-negative scalar curvature},
        date={2022},
     journal={Communications in Contemporary Mathematics},
      volume={0},
      number={0},
       pages={2250030},
      eprint={https://doi.org/10.1142/S0219199722500304},
         url={https://doi.org/10.1142/S0219199722500304},
}

\bib{Colding-shape}{article}{
      author={Colding, Tobias~H.},
       title={Shape of manifolds with positive {R}icci curvature},
        date={1996},
        ISSN={0020-9910},
     journal={Invent. Math.},
      volume={124},
      number={1-3},
       pages={175\ndash 191},
         url={https://doi.org/10.1007/s002220050049},
      review={\MR{1369414}},
}

\bib{Colding-volume}{article}{
      author={Colding, Tobias~H.},
       title={Ricci curvature and volume convergence},
        date={1997},
        ISSN={0003-486X},
     journal={Ann. of Math. (2)},
      volume={145},
      number={3},
       pages={477\ndash 501},
         url={https://doi.org/10.2307/2951841},
      review={\MR{1454700}},
}

\bib{Dai_Wei_Zhang-2018}{article}{
      author={Dai, Xianzhe},
      author={Wei, Guofang},
      author={Zhang, Zhenlei},
       title={Local {S}obolev constant estimate for integral {R}icci curvature
  bounds},
        date={2018},
        ISSN={0001-8708},
     journal={Advances in Mathematics},
      volume={325},
       pages={1\ndash 33},
  url={https://www.sciencedirect.com/science/article/pii/S0001870816311112},
}

\bib{Finster-09}{article}{
      author={Finster, Felix},
       title={A level set analysis of the {W}itten spinor with applications to
  curvature estimates},
        date={2009},
        ISSN={1073-2780},
     journal={Math. Res. Lett.},
      volume={16},
      number={1},
       pages={41\ndash 55},
      review={\MR{2480559 (2010c:53052)}},
}

\bib{Finster-Kath-02}{article}{
      author={Finster, Felix},
      author={Kath, Ines},
       title={Curvature estimates in asymptotically flat manifolds of positive
  scalar curvature},
        date={2002},
        ISSN={1019-8385},
     journal={Comm. Anal. Geom.},
      volume={10},
      number={5},
       pages={1017\ndash 1031},
      review={\MR{1957661 (2004b:53051)}},
}

\bib{Gao-integral1}{article}{
      author={Gao, L.~Zhiyong},
       title={Convergence of {R}iemannian manifolds; {R}icci and
  {$L^{n/2}$}-curvature pinching},
        date={1990},
        ISSN={0022-040X},
     journal={J. Differential Geom.},
      volume={32},
      number={2},
       pages={349\ndash 381},
         url={http://projecteuclid.org/euclid.jdg/1214445311},
      review={\MR{1072910}},
}

\bib{Gromov-Lawson-torus}{article}{
      author={Gromov, Mikhael},
      author={Lawson, H.~Blaine, Jr.},
       title={Spin and scalar curvature in the presence of a fundamental group.
  {I}},
        date={1980},
        ISSN={0003-486X},
     journal={Ann. of Math. (2)},
      volume={111},
      number={2},
       pages={209\ndash 230},
         url={http://dx.doi.org.memex.lehman.cuny.edu:2048/10.2307/1971198},
      review={\MR{569070}},
}

\bib{GroD}{article}{
      author={Gromov, Misha},
       title={Dirac and {P}lateau billiards in domains with corners},
        date={2014},
        ISSN={1895-1074},
     journal={Cent. Eur. J. Math.},
      volume={12},
      number={8},
       pages={1109\ndash 1156},
  url={http://dx.doi.org.memex.lehman.cuny.edu:2048/10.2478/s11533-013-0399-1},
      review={\MR{3201312}},
}

\bib{GT}{book}{
      author={Gilbarg, David},
      author={Trudinger, Neil~S.},
       title={Elliptic partial differential equations of second order},
      series={Classics in mathematics},
   publisher={Springer-Verlag},
        date={1998},
        ISBN={3-540-41160-7},
}

\bib{Heinonen-2001}{book}{
      author={Heinonen, Juha},
       title={Lectures on analysis on metric spaces},
     edition={1},
      series={Universitext},
   publisher={Springer New York},
        date={2001},
}

\bib{Huisken-Ilmanen}{article}{
      author={Huisken, Gerhard},
      author={Ilmanen, Tom},
       title={The inverse mean curvature flow and the {R}iemannian {P}enrose
  inequality},
        date={2001},
        ISSN={0022-040X},
     journal={J. Differential Geom.},
      volume={59},
      number={3},
       pages={353\ndash 437},
      review={\MR{MR1916951 (2003h:53091)}},
}

\bib{Huang-Lee-Perales}{article}{
      author={Huang, Lan-Hsuan},
      author={Lee, Dan~A.},
      author={Perales, Raquel},
       title={Intrinsic flat convergence of points and applications to
  stability of the positive mass theorem},
        date={2022},
     journal={Annales Henri Poincare},
      volume={38},
}

\bib{HLS}{article}{
      author={Huang, Lan-Hsuan},
      author={Lee, Dan~A.},
      author={Sormani, Christina},
       title={Intrinsic flat stability of the positive mass theorem for
  graphical hypersurfaces of {E}uclidean space},
        date={2017},
        ISSN={0075-4102},
     journal={J. Reine Angew. Math.},
      volume={727},
       pages={269\ndash 299},
         url={https://doi.org/10.1515/crelle-2015-0051},
      review={\MR{3652253}},
}

\bib{KKL-2021}{misc}{
      author={Kazaras, Demetre},
      author={Khuri, Marcus},
      author={Lee, Dan},
       title={Stability of the positive mass theorem under {R}icci curvature
  lower bounds},
   publisher={arXiv},
        date={2021},
         url={https://arxiv.org/abs/2111.05202},
}

\bib{li_2012}{book}{
      author={Li, Peter},
       title={Geometric analysis},
      series={Cambridge Studies in Advanced Mathematics},
   publisher={Cambridge University Press},
        date={2012},
}

\bib{LNN}{article}{
      author={Lee, Man-Chun},
      author={Naber, Aaron},
      author={Neumayer, Robin},
       title={$d_p$ convergence and $\epsilon$-regularity theorems for entropy
  and scalar curvature lower bounds},
        date={2021},
      eprint={https://arxiv.org/abs/2010.15663},
}

\bib{LeeSormani1}{article}{
      author={Lee, Dan~A.},
      author={Sormani, Christina},
       title={{S}tability of the positive mass theorem for rotationally
  symmetric riemannian manifolds},
        date={2014},
     journal={Journal fur die Riene und Angewandte Mathematik (Crelle's
  Journal)},
      volume={686},
}

\bib{Petersen-2006}{book}{
      author={Petersen, Peter},
       title={Riemannian geometry},
     edition={2},
   publisher={Springer New York, NY},
        date={2006},
}

\bib{Petersen-97}{article}{
      author={Petersen, Peter},
       title={Convergence theorems in {R}iemannian geometry},
        date={1997},
     journal={MSRI Publications-Comparison Geometry},
      volume={30},
}

\bib{Petersen-Survey}{incollection}{
      author={Petersen, Peter},
       title={Convergence theorems in {R}iemannian geometry},
        date={1997},
   booktitle={Comparison geometry ({B}erkeley, {CA}, 1993--94)},
      series={Math. Sci. Res. Inst. Publ.},
      volume={30},
   publisher={Cambridge Univ. Press, Cambridge},
       pages={167\ndash 202},
         url={https://doi.org/10.2977/prims/1195166127},
      review={\MR{1452874}},
}

\bib{PKP19}{article}{
      author={Pacheco, Armando J.~Cabrera},
      author={Ketterer, Christian},
      author={Perales, Raquel},
       title={Stability of graphical tori with almost nonnegative scalar
  curvature},
        date={2020},
     journal={Calc. Var.},
      volume={59},
      number={134},
}

\bib{Petersen-Sprouse-integral}{article}{
      author={Petersen, Peter},
      author={Sprouse, Chadwick},
       title={Integral curvature bounds, distance estimates and applications},
        date={1998},
        ISSN={0022-040X},
     journal={J. Differential Geom.},
      volume={50},
      number={2},
       pages={269\ndash 298},
         url={http://projecteuclid.org/euclid.jdg/1214461171},
      review={\MR{1684981}},
}

\bib{Petersen_Wei-2001}{article}{
      author={Petersen, Peter},
      author={Wei, Guofang},
       title={Analysis and geometry on manifolds with integral {R}icci
  curvature bounds. ii},
        date={2001},
        ISSN={00029947},
     journal={Transactions of the American Mathematical Society},
      volume={353},
      number={2},
       pages={457\ndash 478},
         url={http://www.jstor.org/stable/221771},
}

\bib{Petersen-Wei-integral2}{article}{
      author={Petersen, Peter},
      author={Wei, Guofang},
       title={Analysis and geometry on manifolds with integral {R}icci
  curvature bounds. {II}},
        date={2001},
        ISSN={0002-9947},
     journal={Trans. Amer. Math. Soc.},
      volume={353},
      number={2},
       pages={457\ndash 478},
         url={https://doi.org/10.1090/S0002-9947-00-02621-0},
      review={\MR{1709777}},
}

\bib{Petersen-Wei-integral1}{article}{
      author={Petersen, P.},
      author={Wei, G.},
       title={Relative volume comparison with integral curvature bounds},
        date={1997},
        ISSN={1016-443X},
     journal={Geom. Funct. Anal.},
      volume={7},
      number={6},
       pages={1031\ndash 1045},
         url={https://doi.org/10.1007/s000390050036},
      review={\MR{1487753}},
}

\bib{Petersen-Wei-1997}{article}{
      author={Petersen, Peter},
      author={Wei, Guofang},
       title={Relative volume comparison with integral curvature bounds},
        date={1997},
     journal={Geometric and Functional Analysis},
      volume={7},
       pages={1031\ndash 1045},
}

\bib{Shanmugalingam-2000}{article}{
      author={Shanmugalingam, Nageswari},
       title={Newtonian spaces: An extension of {S}obolev spaces to metric
  measure spaces},
        date={2000},
     journal={Revista Matem{\'a}tica Iberoamericana},
      volume={16},
       pages={243\ndash 279},
}

\bib{Sormani-scalar}{incollection}{
      author={Sormani, Christina},
       title={Scalar curvature and intrinsic flat convergence},
        date={2017},
   booktitle={Measure theory in non-smooth spaces},
      editor={Gigli, Nicola},
   publisher={De Gruyter Press},
       pages={288\ndash 338},
}

\bib{ScalarSurvey-Sormani}{incollection}{
      author={Sormani, C.},
       title={Conjectures on convergence and scalar curvature},
        date={2022},
   booktitle={Chapter in volume 2 of perpectives in scalar curvature},
      editor={Gromov, M.},
      editor={Lawson, B.},
   publisher={World Scientific},
}

\bib{Sormani-Stavrov-1}{article}{
      author={Sormani, C},
      author={Stavrov, I},
       title={Geometrostatic manifolds of small adm mass},
        date={2019},
     journal={Communications on Pure and Applied Mathematics},
      volume={72},
}

\bib{Stern-19}{article}{
      author={Stern, Daniel},
       title={Scalar curvature and harmonic maps to {$S^1$}},
        date={2020},
     journal={(accepted) Journal of Differential Geometry},
      eprint={https://arxiv.org/abs/1908.09754},
         url={https://arxiv.org/abs/1908.09754},
}

\bib{Strzelecki-06}{article}{
      author={Strzelecki, P.},
       title={Gagliardo–{N}irenberg inequalities with a {BMO} term},
        date={2006},
     journal={Bulletin of the London Mathematical Society},
      volume={38},
      number={2},
       pages={294\ndash 300},
  eprint={https://londmathsoc.onlinelibrary.wiley.com/doi/pdf/10.1112/S0024609306018169},
  url={https://londmathsoc.onlinelibrary.wiley.com/doi/abs/10.1112/S0024609306018169},
}

\bib{Schoen-Yau-min-surf}{article}{
      author={Schoen, R.},
      author={Yau, Shing~Tung},
       title={Existence of incompressible minimal surfaces and the topology of
  three-dimensional manifolds with nonnegative scalar curvature},
        date={1979},
        ISSN={0003-486X},
     journal={Ann. of Math. (2)},
      volume={110},
      number={1},
       pages={127\ndash 142},
         url={http://dx.doi.org.memex.lehman.cuny.edu:2048/10.2307/1971247},
      review={\MR{541332}},
}

\bib{SchoenYauPMT}{article}{
      author={Schoen, Richard},
      author={Yau, Shing~Tung},
       title={On the proof of the positive mass conjecture in general
  relativity},
        date={1979},
        ISSN={0010-3616},
     journal={Comm. Math. Phys.},
      volume={65},
      number={1},
       pages={45\ndash 76},
         url={http://projecteuclid.org/euclid.cmp/1103904790},
      review={\MR{526976}},
}

\bib{Witten-PMT}{article}{
      author={Witten, Edward},
       title={A new proof of the positive energy theorem},
        date={1981},
        ISSN={0010-3616},
     journal={Comm. Math. Phys.},
      volume={80},
      number={3},
       pages={381\ndash 402},
  url={http://projecteuclid.org.memex.lehman.cuny.edu:2048/getRecord?id=euclid.cmp/1103919981},
      review={\MR{626707 (83e:83035)}},
}

\bib{DYang-integral2}{article}{
      author={Yang, Deane},
       title={Convergence of {R}iemannian manifolds with integral bounds on
  curvature. {I}},
        date={1992},
        ISSN={0012-9593},
     journal={Ann. Sci. \'{E}cole Norm. Sup. (4)},
      volume={25},
      number={1},
       pages={77\ndash 105},
         url={http://www.numdam.org/item?id=ASENS_1992_4_25_1_77_0},
      review={\MR{1152614}},
}

\bib{DYang-integral3}{article}{
      author={Yang, Deane},
       title={Convergence of {R}iemannian manifolds with integral bounds on
  curvature. {II}},
        date={1992},
        ISSN={0012-9593},
     journal={Ann. Sci. \'{E}cole Norm. Sup. (4)},
      volume={25},
      number={2},
       pages={179\ndash 199},
         url={http://www.numdam.org/item?id=ASENS_1992_4_25_2_179_0},
      review={\MR{1169351}},
}

\bib{DYang-integral1}{article}{
      author={Yang, Deane},
       title={Riemannian manifolds with small integral norm of curvature},
        date={1992},
        ISSN={0012-7094},
     journal={Duke Math. J.},
      volume={65},
      number={3},
       pages={501\ndash 510},
         url={https://doi.org/10.1215/S0012-7094-92-06519-7},
      review={\MR{1154180}},
}

\end{biblist}
\end{bibdiv}

\end{document}